\documentclass[reqno,12pt]{amsart}

\pdfoutput=1

\usepackage{color} 
\usepackage{ifpdf}
\ifpdf
    \usepackage[pdftex]{graphicx}
    \usepackage[pdftex]{hyperref}
    \hypersetup{
        unicode=false,          
        pdftoolbar=true,        
        pdfmenubar=true,        
        pdffitwindow=false,     
        pdfstartview={FitH},    
        pdftitle={MCP Article},      
        pdfauthor={Michael Holst},   
        pdfsubject={Mathematics},    
        pdfcreator={Michael Holst},  
        pdfproducer={Michael Holst}, 
        pdfkeywords={PDE, analysis, mathematical physics}, 
        pdfnewwindow=true,      
        colorlinks=true,        
        linkcolor=red,          
        citecolor=blue,         
        filecolor=magenta,      
        urlcolor=cyan           
    }

\else
    \usepackage{graphicx}
    \usepackage{pstricks}

\fi

\usepackage{times}
\usepackage{amsfonts}

\usepackage{amsmath}
\usepackage{amsthm}
\usepackage{amssymb}
\usepackage{amsbsy}
\usepackage{amscd}

\usepackage{enumerate}
\usepackage{verbatim}
\usepackage{subfigure}

\newtheorem{theorem}{Theorem}[section]

\newtheorem{lemma}[theorem]{Lemma}

\numberwithin{equation}{section}  

  \newcounter{mnote}
  \let\oldmarginpar\marginpar
    \renewcommand\marginpar[1]{\-\oldmarginpar[\raggedleft\footnotesize #1]%
    {\raggedright\footnotesize #1}}

\definecolor{myblue}{rgb}{0.2,0.2,0.7}
\definecolor{mygreen}{rgb}{0,0.6,0}
\definecolor{mycyan}{rgb}{0,0.6,0.6}
\definecolor{myred}{rgb}{0.9,0.2,0.2}
\definecolor{mymagenta}{rgb}{0.9,0.2,0.9}
\definecolor{mywhite}{rgb}{1.0,1.0,1.0}
\definecolor{myblack}{rgb}{0.0,0.0,0.0}

\newcommand{\beq}{\begin{equation}}
\newcommand{\eeq}{\end{equation}}
\newcommand{\beqa}{\begin{eqnarray}}
\newcommand{\eeqa}{\end{eqnarray}}
\begin{document}

\title[APPLICATIONS OF DD AND PUM IN PHYSICS AND GEOMETRY]
      {Applications of Domain Decomposition and Partition of Unity Methods in Physics and Geometry}


\author[M. Holst]{Michael Holst}
\email{mholst@math.ucsd.edu}

\address{Department of Mathematics,
         University of California, San Diego
         9500 Gilman Drive, Dept. 0112,
         La Jolla, CA 92093-0112 USA}

\thanks{The author was supported in part by NSF CAREER Award~9875856, by NSF Grants~0225630, 0208449, 0112413, and by a UCSD Hellman Fellowship.}

\date{September 25, 2002}


\begin{abstract}
We consider a class of adaptive multilevel domain
decomposition-like algorithms, built from a combination of adaptive multilevel
finite element, domain decomposition, and partition of unity methods.  These 
algorithms have several interesting features such as very low communication 
requirements, and they inherit a simple and elegant approximation theory 
framework from partition of unity methods.  They are also very easy to use 
with highly complex sequential adaptive finite element packages, requiring 
little or no modification of the underlying sequential finite element 
software.  The parallel algorithm can be implemented as a simple loop which 
starts off a sequential local adaptive solve on a collection of processors
simultaneously.  We first review the Partition of Unity Method (PUM) of
Babu\v{s}ka and Melenk, and outline the PUM approximation theory framework.
We then describe a variant we refer to here as the Parallel Partition of 
Unity Method (PPUM), which is a combination of the Partition of Unity Method
with the parallel adaptive algorithm of Bank and Holst.  We then derive two 
global error estimates for PPUM, by exploiting the PUM analysis framework it 
inherits, and by employing some recent local estimates of Xu and Zhou.  We 
then discuss a duality-based variant of PPUM which is more appropriate for 
certain applications, and we derive a suitable variant of the PPUM 
approximation theory framework.  Our implementation of PPUM-type algorithms 
using the {\sc FETK} and {\sc MC} software packages is described.  We then 
present a short numerical example involving the Einstein constraints arising 
in gravitational wave models.
\end{abstract}

\maketitle


\vspace*{-1.2cm}
{\footnotesize
\tableofcontents
}

\section[Introduction]
        {Introduction}
\label{Holst:sec:intro}

In this article we consider a class of adaptive multilevel domain
decomposition-like algorithms, built from a combination of adaptive multilevel
finite element, domain decomposition, and partition of unity methods.
These algorithms have several interesting features such as very
low communication requirements, and they inherit a simple and elegant
approximation theory framework from partition of unity methods.
They are also very easy to use with highly complex sequential adaptive
finite element packages, requiring little or no modification of the
underlying sequential finite element software.
The parallel algorithm can be implemented as a simple loop which starts off
a sequential local adaptive solve on a collection of processors
simultaneously.

We first review the Partition of Unity Method (PUM) of
Babu\v{s}ka and Melenk in Section~\ref{Holst:sec:pum},
and outline the PUM approximation theory framework.
In Section~\ref{Holst:sec:ppum}, we describe a variant we
refer to here as the Parallel Partition of Unity Method (PPUM),
which is a combination of the Partition of Unity Method
with the parallel adaptive algorithm from~\cite{Bank:2000:NPP}.
We then derive two global error estimates for PPUM, by exploiting
the PUM analysis framework it inherits, and by employing some recent
local estimates of Xu and Zhou~\cite{Xu:2000:LPF}.
We then discuss a duality-based variant of PPUM in
Section~\ref{Holst:sec:duality} which is more appropriate for certain
applications, and we derive a suitable variant of the 
PPUM approximation theory framework.
Our implementation of PPUM-type algorithms using the {\sc FEtk} and {\sc MC}
software packages is described in Section~\ref{Holst:sec:mc}.
We then present a short numerical example in Section~\ref{Holst:sec:example1}
involving the Einstein constraints arising in gravitational wave models.

\section[The Partition of Unity Method (PUM)]
        {The Partition of Unity Method (PUM) of Babu\v{s}ka and Melenk}
\label{Holst:sec:pum}

We first briefly review the partition of unity method (PUM) of
Babu\v{s}ka and Melenk~\cite{Babuska:1997:PUF}.
Let $\Omega \subset \mathbb{R}^d$ be an open set and let $\{ \Omega_i \}$
be an open cover of $\Omega$ with a bounded local overlap property:
For all $x \in \Omega$, there exists a constant $M$ such that
\begin{equation}
   \label{Holst:eqn:pum_overlap}
\sup_i \{~ i ~|~ x \in \Omega_i ~\} \le M.
\end{equation}
A Lipschitz {\em partition of unity} $\{ \phi_i \}$
subordinate to the cover $\{ \Omega_i \}$ satisfies the
following five conditions:
\begin{eqnarray}
\sum_i \phi_i(x) &\equiv& 1, ~~~ \forall x \in \Omega, 
\label{Holst:eqn:pum_pou1}
    \\
\phi_i &\in& C^k(\Omega) ~~~ \forall i, ~~~ (k \ge 0),
\label{Holst:eqn:pum_pou2}
    \\
\text{supp}~ \phi_i &\subset& \overline{\Omega}_i, ~~~ \forall i,
\label{Holst:eqn:pum_pou3}
    \\
\| \phi_i \|_{L^{\infty}(\Omega)} &\le& C_{\infty}, ~~~ \forall i,
\label{Holst:eqn:pum_pou4}
    \\
\| \nabla \phi_i \|_{L^{\infty}(\Omega)} 
   &\le& \frac{C_G}{\text{diam}(\Omega_i)}, ~~~ \forall i.
\label{Holst:eqn:pum_pou5}
\end{eqnarray}
Several explicit constructions of partitions of unity
satisfying~(\ref{Holst:eqn:pum_pou1})--(\ref{Holst:eqn:pum_pou5}) exist.
The simplest construction in the case of a polygon
$\Omega \subset \mathbb{R}^d$ employs global $C^0$ piecewise linear finite
element basis functions defined on a simplex mesh
subdivision $\mathcal{S}$ of $\Omega$.
The $\{\Omega_i\}$ are first built by first constructing a disjoint
partitioning $\{\Omega_i^{\circ}\}$ of $\mathcal{S}$ using e.g. spectral or
inertial bisection~\cite{Bank:2000:NPP}.
Each of the disjoint $\Omega_i^{\circ}$ are extended to define $\Omega_i$
by considering all boundary vertices of $\Omega_i^{\circ}$;
all simplices of neighboring $\Omega_j^{\circ}$, $j\ne i$
which are contained in the boundary vertex $1$-rings of $\Omega_i^{\circ}$
are added to $\Omega_i^{\circ}$ to form $\Omega_i$.
This procedure produces the smallest overlap for the $\{ \Omega_i \}$,
such that the properties~(\ref{Holst:eqn:pum_pou1})--(\ref{Holst:eqn:pum_pou4}) are satisfied
by the resulting $\{\phi_i\}$ built from the nodal $C^0$ piecewise linear
finite element basis functions.
Property~(\ref{Holst:eqn:pum_pou5}) is also satisfied, but $C_G$ will depend on the
diameter of the overlap simplices.
More sophisticated constructions with superior properties are possible;
see e.g.~\cite{Griebel:2000:PPU,Shepard:1968:TDI}.

The {\em partition of unity method (PUM)}
builds an approximation $u_{ap} = \sum_i \phi_i v_i$ where the
$v_i$ are taken from the local approximation spaces:
\begin{equation}
   \label{Holst:eqn:pum_local_spaces}
V_i \subset C^k(\Omega \cap \Omega_i)
     \subset H^1(\Omega \cap \Omega_i), ~~~ \forall i, ~~~ (k \ge 0).
\end{equation}
The following simple lemma makes possible several useful results.
\begin{lemma} \label{Holst:lemma:pum}
Let $w,w_i \in H^1(\Omega)$ with
$\text{supp}~ w_i \subseteq \overline{\Omega \cap \Omega_i}$.
Then
\begin{eqnarray*}
\sum_i \|w\|_{H^k(\Omega_i)}^2 
   &\le& M \|w\|_{H^k(\Omega)}^2, ~~~~ k=0,1
\\
\| \sum_i w_i \|_{H^k(\Omega)}^2 
   &\le& M \sum_i \|w_i\|_{H^k(\Omega \cap \Omega_i)}^2, ~~~~ k=0,1
\end{eqnarray*}
\end{lemma}
\begin{proof}
The proof follows from~(\ref{Holst:eqn:pum_overlap})
and~(\ref{Holst:eqn:pum_pou1})--(\ref{Holst:eqn:pum_pou5});
see~\cite{Babuska:1997:PUF}.
\end{proof}
The basic approximation properties of PUM following from~\ref{Holst:lemma:pum}
are as follows.
\begin{theorem} [Babu\v{s}ka and Melenk~\cite{Babuska:1997:PUF}]
   \label{Holst:thm:pum}
If the local spaces $V_i$ have the
following approximation properties:
\begin{eqnarray*}
\|u-v_i\|_{L^2(\Omega \cap \Omega_i)} &\le& \epsilon_0(i),
   ~~~ \forall i, \\
\|\nabla(u-v_i)\|_{L^2(\Omega \cap \Omega_i)} &\le& \epsilon_1(i), 
   ~~~ \forall i,
\end{eqnarray*}
then the following {\em a~priori} global error estimates hold:
\begin{eqnarray*}
\|u-u_{ap}\|_{L^2(\Omega)}
   &\le& \sqrt{M} C_{\infty} \left( \sum_i \epsilon_0^2(i) \right)^{1/2}, \\
\|\nabla(u-u_{ap})\|_{L^2(\Omega)}
   &\le& \sqrt{2M} \left( 
     \sum_i \left( \frac{C_G}{\text{diam}(\Omega_i)} \right)^2
         \epsilon_1^2(i)
     + C_{\infty}^2 \epsilon_0^2(i) \right)^{1/2}.
\end{eqnarray*}
\end{theorem}
\begin{proof}
This follows from Lemma~\ref{Holst:lemma:pum} by taking
$u-u_{ap} = \sum_i \phi_i(u-v_i)$ and then by using
$w_i = \phi_i(u-v_i)$ in Lemma~\ref{Holst:lemma:pum}.
\end{proof}

Consider now the following linear elliptic problem:
\begin{equation}
   \label{Holst:eqn:pde_strong}
\begin{array}{rcl}
- \nabla \cdot (a \nabla u) &=& f ~\text{in}~ \Omega, \\
                          u &=& 0 ~\text{on}~ \partial \Omega,
\end{array}
\end{equation}
where $a_{ij} \in W^{1,\infty}(\Omega)$, $f \in L^2(\Omega)$,
$a_{ij} \xi_i \xi_j \ge a_0 > 0$, $\forall \xi_i \ne 0$,
where $\Omega \subset \mathbb{R}^d$ is a convex polyhedral domain.
A weak formulation is:
\begin{equation}
   \label{Holst:eqn:pde_weak}
\text{Find}~u \in H^1_0(\Omega) ~\text{such~that}~
     \langle F(u),v \rangle = 0, ~~~\forall v \in H^1_0(\Omega),
\end{equation}
where
$$
\langle F(u),v \rangle = \int_{\Omega} a \nabla u \cdot \nabla v ~dx
                       - \int_{\Omega} f v ~dx.
$$
A general Galerkin approximation is the solution to the subspace problem:
\begin{equation}
   \label{Holst:eqn:pde_galerkin}
\text{Find}~ u_{ap} \in V \subset H^1_0(\Omega) ~\text{s.t.}~
    \langle F(u_{ap}),v \rangle = 0, ~\forall v \in V \subset H^1_0(\Omega).
\end{equation}
With PUM, the subspace $V$ for the Galerkin approximation
is taken to be the globally coupled {\em PUM space}
(cf.~\cite{Griebel:2000:PPU}):
$$
V = \left\{~ v ~|~ v = \sum_i \phi_i v_i, ~ v_i \in V_i ~\right\}
     \subset H^1(\Omega),
$$
If error estimates are available for the quality of the local solutions
produced in the local spaces, then the PUM approximation theory
framework given in Theorem~\ref{Holst:thm:pum} guarantees a global
solution quality.

\section[A Parallel Partition of Unity Method (PPUM)]
        {A Parallel Partition of Unity Method (PPUM)}
\label{Holst:sec:ppum}

A new approach to the use of parallel computers with adaptive finite
element methods was presented recently in~\cite{Bank:2000:NPP}.
The following variant of the algorithm in~\cite{Bank:2000:NPP} is
described in~\cite{Holst:2001:ANT},
which we refer to as the {\em Parallel Partition of Unity Method} (or PPUM).
This variant replaces the final global smoothing iteration
in~\cite{Bank:2000:NPP}
with a reconstruction based on Babu\v{s}ka and Melenk's original
Partition of Unity Method~\cite{Babuska:1997:PUF}, which provides
some additional approximation theory structure.

\noindent
{\em
{\bf Algorithm}
{\em (PPUM - Parallel Partition of
             Unity Method~\cite{Bank:2000:NPP,Holst:2001:ANT})}
\begin{center}
\begin{minipage}{5.0in}
\begin{enumerate}
\item Discretize and solve the problem using a global coarse mesh.
\item Compute {\em a posteriori} error estimates using the coarse solution, and decompose the mesh to achieve equal error using weighted spectral or inertial bisection.
\item Give the entire mesh to a collection of processors, where each processor will perform a completely independent multilevel adaptive solve, restricting local refinement to only an assigned portion of the domain.  The portion of the domain assigned to each processor coincides with one of the domains produced by spectral bisection with some overlap (produced by conformity algorithms, or by explicitly enforcing substantial overlap).  When a processor has reached an error tolerance locally, computation stops on that processor.
\item Combine the independently produced solutions using a partition of unity subordinate to the overlapping subdomains.
\end{enumerate}
\end{minipage}
\end{center}
}

While the PPUM algorithm seems to ignore the global coupling of the
elliptic problem, recent results on local error estimation~\cite{Xu:2000:LPF},
as well as some not-so-recent results on interior
estimates~\cite{Nitsche:1974:IER},
support this as provably good in some sense.
The principle idea underlying the results
in~\cite{Nitsche:1974:IER,Xu:2000:LPF} is that while
elliptic problems are globally coupled, this global coupling is essentially
a ``low-frequency'' coupling, and can be handled on the initial mesh which
is much coarser than that required for approximation accuracy considerations.
This idea has been exploited, for example, in~\cite{Xu:1996:TGD,Xu:2000:LPF}, 
and is
why the construction of a coarse problem in overlapping domain decomposition
methods is the key to obtaining convergence rates which are independent
of the number of subdomains (c.f.~\cite{Xu:1992:IMS}).
An example showing the types of local refinements that occur within each
subdomain is depicted in Figure~\ref{Holst:fig:parallel}.
\begin{figure}[ht]
\begin{center}
\mbox{\includegraphics[width=0.3\textwidth]{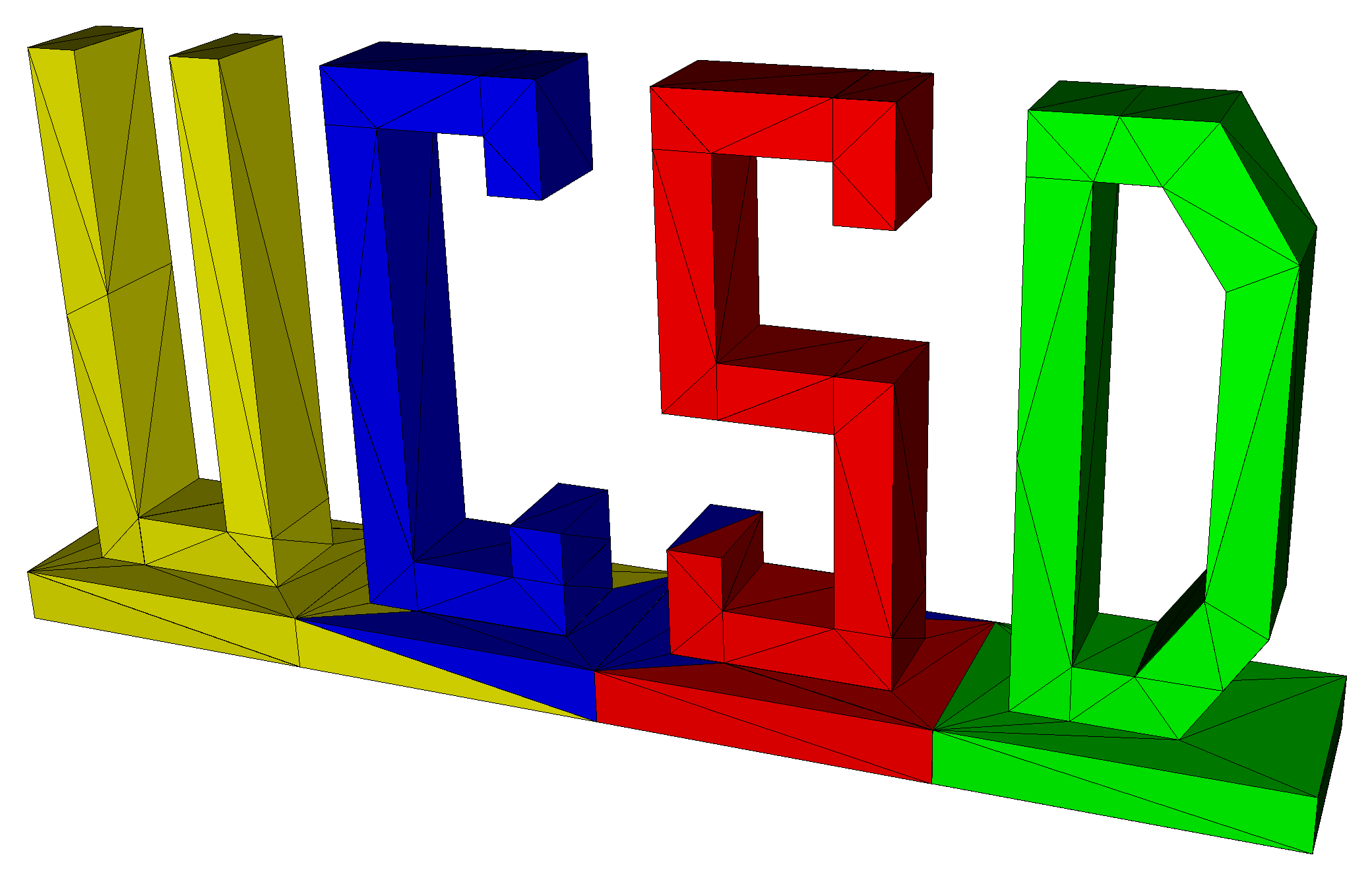}}
\mbox{\includegraphics[width=0.3\textwidth]{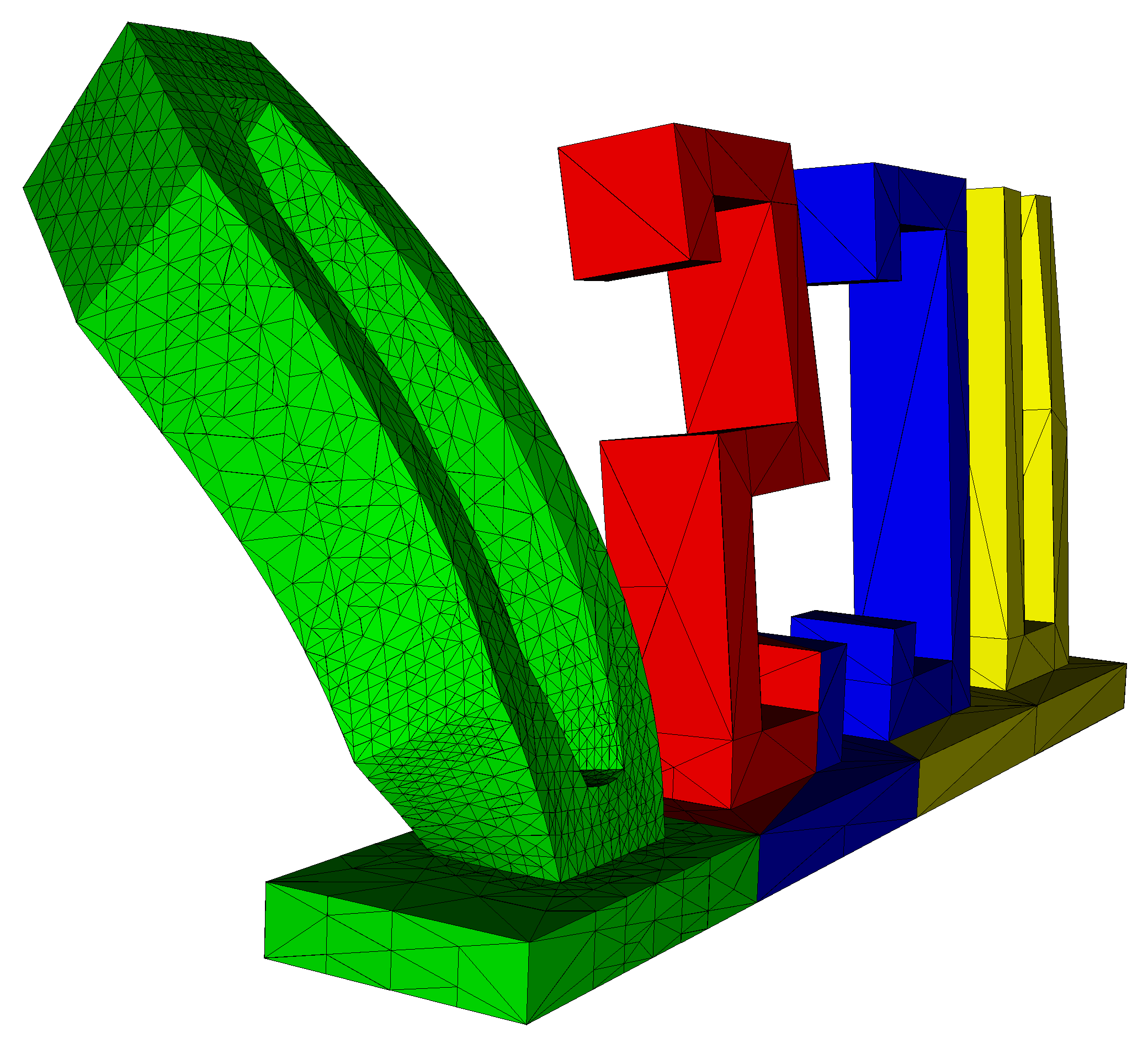}}
\mbox{\includegraphics[width=0.3\textwidth]{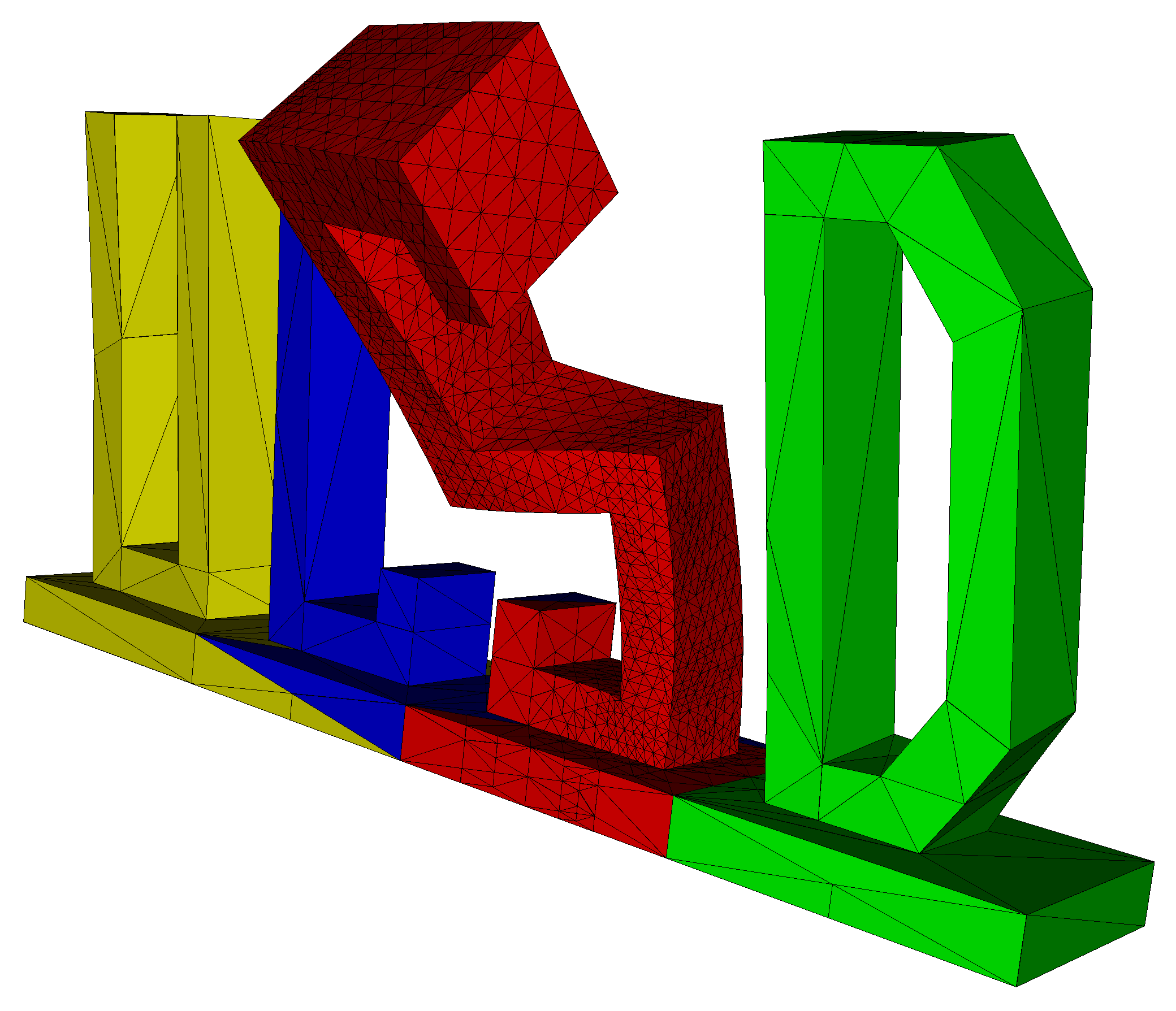}}
\end{center}
\caption{Example showing the types of local refinements created by PPUM.}
  \label{Holst:fig:parallel}
\end{figure}

To illustrate how PPUM can produce a quality global solution,
we will give a global error estimate for PPUM solutions.
This analysis can also be found in~\cite{Holst:2001:ANT}.
We can view PPUM as building a PUM approximation
$u_{pp} = \sum_i \phi_i v_i$ where the $v_i$ are taken from the
local spaces:
\begin{equation}
   \label{Holst:eqn:ppum_local_spaces}
V_i = \mathcal{X}_i V_i^g \subset C^k(\Omega \cap \Omega_i)
     \subset H^1(\Omega \cap \Omega_i), ~~~ \forall i, ~~~ (k \ge 0),
\end{equation}
where $\mathcal{X}_i$ is the characteristic function for $\Omega_i$, and where
\begin{equation}
   \label{Holst:eqn:ppum_local_spaces2}
V_i^g \subset C^k(\Omega) \subset H^1(\Omega), ~~~ \forall i, ~~~ (k \ge 0).
\end{equation}
In PPUM, the global spaces $V_i^g$
in~(\ref{Holst:eqn:ppum_local_spaces})--(\ref{Holst:eqn:ppum_local_spaces2})
are built from locally enriching an initial coarse global space $V_0$
by locally adapting the finite element mesh on which $V_0$ is built.
(This is in contrast to classical overlapping Schwarz
methods where local spaces are often built through enrichment of
$V_0$ by locally adapting the mesh on which $V_0$ is built,
and then removing the portions of the mesh exterior to the
adapted region.)
The PUM space $V$ is then
\begin{eqnarray*}
V &=& \left\{~ v ~|~ v = \sum_i \phi_i v_i, ~ v_i \in V_i ~\right\} \\
  &=& \left\{~ v ~|~ v = \sum_i \phi_i \mathcal{X}_i v_i^g = \sum_i \phi_i v_i^g, 
     ~ v_i^g \in V_i^g ~\right\} \subset H^1(\Omega).
\nonumber
\end{eqnarray*}

In contrast to the approach in PUM where one seeks a global
Galerkin solution in the PUM space as in~(\ref{Holst:eqn:pde_galerkin}),
the PPUM algorithm described here and in~\cite{Holst:2001:ANT} builds a
global approximation $u_{pp}$ to the solution to~(\ref{Holst:eqn:pde_weak})
from decoupled {\em local} Galerkin solutions:
\begin{equation}
   \label{Holst:eqn:ppum_upp}
u_{pp} = \sum_i \phi_i u_i
       = \sum_i \phi_i u_i^g,
\end{equation}
where each $u_i^g$ satisfies:
\begin{equation}
   \label{Holst:eqn:ppum_uig}
\text{Find}~ u_i^g \in V_i^g ~\text{such~that}~
       \langle F(u_i^g),v_i^g \rangle = 0, ~~~\forall v_i^g \in V_i^g.
\end{equation}
We have the following global error estimate for the approximation $u_{pp}$ 
in~(\ref{Holst:eqn:ppum_upp}) built from~(\ref{Holst:eqn:ppum_uig})
using the local PPUM parallel algorithm.
\begin{theorem} \label{Holst:thm:ppum}
Assume the solution to~(\ref{Holst:eqn:pde_strong}) satisfies
$u \in H^{1+\alpha}(\Omega)$, $\alpha > 0$, that
quasi-uniform meshes of sizes $h$ and $H > h$ are used for $\Omega_i^0$
and $\Omega \backslash \Omega_i^0$ respectively,
and that $\text{diam}(\Omega_i) \ge 1/Q > 0 ~~\forall i$.
If the local solutions are built from $C^0$ piecewise linear finite elements,
then the global solution $u_{pp}$ in~(\ref{Holst:eqn:ppum_upp}) produced by
Algorithm PPUM satisfies the following global error bounds:
\begin{eqnarray*}
\|u-u_{pp}\|_{L^2(\Omega)}
   &\le& \sqrt{PM} C_{\infty} 
         \left( C_1 h^{\alpha} + C_2 H^{1+\alpha} \right), \\
\|\nabla(u-u_{pp})\|_{L^2(\Omega)}
   &\le& \sqrt{2PM (Q^2 C_G^2 + C_{\infty}^2)}
         \left( C_1 h^{\alpha} + C_2 H^{1+\alpha} \right),
\end{eqnarray*}
where $P = $ number of local spaces $V_i$.
Further, if $H \le h^{\alpha/(1+\alpha)}$ then:
\begin{eqnarray*}
\|u-u_{pp}\|_{L^2(\Omega)}
   &\le& \sqrt{PM} C_{\infty} \max\{C_1,C_2\} h^{\alpha}, \\
\|\nabla(u-u_{pp})\|_{L^2(\Omega)}
    &\le& \sqrt{2PM (Q^2 C_G^2 + C_{\infty}^2)} \max\{C_1,C_2\} h^{\alpha},
\end{eqnarray*}
so that the solution produced by Algorithm PPUM
is of optimal order in the $H^1$-norm.
\end{theorem}
\begin{proof}
Viewing PPUM as a PUM gives access to the {\em a~priori}
estimates in Theorem~\ref{Holst:thm:pum};
these require local estimates of the form:
\begin{eqnarray*}
\|u-u_i\|_{L^2(\Omega \cap \Omega_i)} 
= \|u-u_i^g\|_{L^2(\Omega \cap \Omega_i)} &\le& \epsilon_0(i), \\
\|\nabla(u-u_i)\|_{L^2(\Omega \cap \Omega_i)} 
= \|\nabla(u-u_i^g)\|_{L^2(\Omega \cap \Omega_i)} &\le& \epsilon_1(i).
\end{eqnarray*}
Such local {\em a~priori} estimates are available
for problems of the
form~(\ref{Holst:eqn:pde_strong})~\cite{Nitsche:1974:IER,Xu:2000:LPF}.
They can be shown to take the following form:
$$
\| u - u_i^g \|_{H^1(\Omega_i \cap \Omega)} \le C \left(
    \inf_{v_i^0 \in V_i^0} \|u - v_i^0\|_{H^1(\Omega_i^0 \cap \Omega)}
  + \| u - u_i^g \|_{L^2(\Omega)} \right)
$$
where
$$
   V_i^0 \subset C^k(\Omega_i^0 \cap \Omega) \subset H^1(\Omega_i \cap \Omega),
$$
and where
$$
   \Omega_i \subset \subset \Omega_i^0,
   \ \ \ \ \ \Omega_{ij} = \Omega_i^0 \bigcap \Omega_i^0,
   \ \ \ \ \ |\Omega_{ij} | \approx | \Omega_i | \approx | \Omega_j |.
$$
Since we assume $u \in H^{1+\alpha}(\Omega)$, $\alpha > 0$,
and since quasi-uniform meshes of sizes $h$ and $H > h$ are used
for $\Omega_i^0$ and $\Omega \backslash \Omega_i^0$ respectively, we have:
\begin{eqnarray*}
\| u - u_i^g \|_{H^1(\Omega_i \cap \Omega)}
&=& \left(
\| u - u_i^g \|_{L^2(\Omega_i \cap \Omega)}^2
+ \| \nabla( u - u_i^g) \|_{L^2(\Omega_i \cap \Omega)}^2
\right)^{1/2}
\nonumber \\
&\le& C_1 h^{\alpha} + C_2 H^{1 + \alpha}.
\end{eqnarray*}
I.e., in this setting we can use
$\epsilon_0(i) = \epsilon_1(i) = C_1 h^{\alpha} + C_2 H^{1 + \alpha}$.
The {\em a~priori} PUM estimates in Theorem~\ref{Holst:thm:pum} then become:
\begin{eqnarray*}
\|u-u_{pp}\|_{L^2(\Omega)}
   &\le& \sqrt{M} C_{\infty} \left(
           \sum_i (C_1 h^{\alpha} + C_2 H^{1+\alpha})^2 \right)^{1/2},
\\
\|\nabla(u-u_{pp})\|_{L^2(\Omega)}
   &\le& \sqrt{2M} 
\nonumber
\end{eqnarray*}
$$
\cdot
\left( \left[
     \sum_i \left( \frac{C_G}{\text{diam}(\Omega_i)} \right)^2
     + C_{\infty}^2 
      \right]
         (C_1 h^{\alpha} + C_2 H^{1+\alpha})^2
         \right)^{1/2}.
$$
If $P = $ number of local spaces $V_i$,
and if $\text{diam}(\Omega_i) \ge 1/Q > 0 ~~\forall i$, this is simply:
\begin{eqnarray*}
\|u-u_{pp}\|_{L^2(\Omega)}
   &\le& \sqrt{PM} C_{\infty} 
         \left( C_1 h^{\alpha} + C_2 H^{1+\alpha} \right), \\
\|\nabla(u-u_{pp})\|_{L^2(\Omega)}
   &\le& \sqrt{2PM (Q^2 C_G^2 + C_{\infty}^2)}
         \left( C_1 h^{\alpha} + C_2 H^{1+\alpha} \right).
\end{eqnarray*}
If $H \le h^{\alpha/(1+\alpha)}$ then $u_{pp}$ from PPUM is
asymptotically as good as a global Galerkin solution when
the error is measured in the $H^1$-norm.
\end{proof}
Local versions of Theorem~\ref{Holst:thm:ppum} appear in~\cite{Xu:2000:LPF} for
a variety of related parallel algorithms.
Note that the local estimates in~\cite{Xu:2000:LPF} hold more generally for
nonlinear versions of~(\ref{Holst:eqn:pde_strong}), so that
Theorem~\ref{Holst:thm:ppum} can be shown to hold in a more general setting.
Finally, it should be noted that improving the estimates in the
$L^2$-norm is not generally possible; the required local estimates simply
do not hold.  Improving the solution quality in the $L^2$-norm generally
requires more global information.
However, for some applications one is more interested in a quality
approximation of the {\em gradient} or the {\em energy} of the solution rather
than to the solution itself.

\section[Duality-based PPUM]
        {Duality-based PPUM}
\label{Holst:sec:duality}

We first briefly review a standard approach to the use of duality
methods in error estimation.
(cf.~\cite{Estep:2002:ASP,Giles:2002:AMP}
for a more complete discussion).
Consider the weak formulation~(\ref{Holst:eqn:pde_weak}) involving
a possibly nonlinear differential operator
$F: H^1_0(\Omega) \mapsto H^{-1}(\Omega)$,
and a Galerkin approximation $u_{ap}$
satisfying~(\ref{Holst:eqn:pde_galerkin}).
If $F \in C^1$, the generalized Taylor expansion exists:
$$
F(u+h) = F(u) + \left\{ \int_0^1 DF(u+\xi h) d\xi \right\} h.
$$
With $e=u-u_{ap}$, and with $F(u)=0$,
leads to the linearized error equation:
$$
F(u_{ap}) = F(u - e) 
          = F(u) + \mathcal{A}(u_{ap}-u) = -\mathcal{A}e,
$$
where the linearization operator $\mathcal{A}$ is defined as:
$$
\mathcal{A} = \int_0^1 DF(u+\xi h) d\xi.
$$
Assume now we are interested in a linear functional of the error
$l(e) = \langle e,\psi \rangle$, where $\psi$ is the (assumed accessible)
Riesz-representer of $l(\cdot)$.
If $\phi \in H^1_0(\Omega)$ is the solution to the linearized dual problem:
$$
\mathcal{A}^T \phi = \psi,
$$
then we can exploit the linearization operator $\mathcal{A}$
and its adjoint $\mathcal{A}^T$ to give the following identity:
\begin{equation}
   \label{Holst:eqn:residual}
\langle e,\psi \rangle = \langle e,\mathcal{A}^T \phi \rangle
   = \langle \mathcal{A}e,\phi \rangle
   = - \langle F(u_{ap}),\phi \rangle.
\end{equation}
If we can compute an approximation 
$\phi_{ap} \in \mathcal{V} \subset H^1_0(\Omega)$ to the linearized dual
problem then we can estimate the error by combining this with the
(computable) residual $F(u_{ap})$:
$$
|\langle e,\psi \rangle | 
= | \langle F(u_{ap}),\phi \rangle |
= | \langle F(u_{ap}),\phi - \phi_{ap} \rangle |,
$$
where the last term is a result of~(\ref{Holst:eqn:pde_galerkin}). 
The term on the right is then estimated locally using assumptions
on the quality of the approximation $\phi_{ap}$ and by various
numerical techniques; cf.~\cite{Estep:2002:ASP}.
The local estimates are then used to drive adaptive mesh refinement.
This type of duality-based error estimation has been shown to be
useful for certain applications in engineering and other areas where
accuracy in a linear functional of the solution is important, but accuracy
in the solution itself is not (cf.~\cite{Giles:2002:AMP}).

Consider now this type of error estimation in the context of
domain decomposition and PPUM.
Given a linear or nonlinear weak formulation as in~(\ref{Holst:eqn:pde_weak}),
we are interested in the solution $u$ as well as in the error in
PPUM approximations $u_{pp}$ as defined
in~(\ref{Holst:eqn:ppum_upp})--(\ref{Holst:eqn:ppum_uig}).
If a global linear functional $l(u-u_{pp})$ of the error $u-u_{pp}$
is of interest rather
than the error itself, then we can formulate a variant of the PPUM 
parallel algorithm which has in some sense a more general approximation
theory framework than that of the previous section.
There are no assumptions beyond solvability of the local problems and
of the global dual problems with localized data, and perhaps some
minimal smoothness assumptions on the dual solution.
In particular, the theory does not require local {\em a~priori}
error estimates;
the local {\em a~priori} estimates are replaced by
solving global dual problem problems
with localized data, and then incorporating the dual solutions explictly
into the {\em a posteriori} error estimate.
As a result, the large overlap assumption needed for the local estimates
in the proof of Theorem~\ref{Holst:thm:ppum} is unnecessary.
Similarly, the large overlap assumption needed to achieve the bounded
gradient property~(\ref{Holst:eqn:pum_pou5}) is no longer needed.

The following result gives a global bound on a linear functional of the
error based on satisfying local computable {\em a posteriori} bounds
involving localized dual problems.
\begin{theorem}
    \label{Holst:thm:ppum_dual}
Let $\{ \phi_i \}$ be a partition of unity subordinate to a cover
$\{ \Omega_i \}$.
If $\psi$ is the Riesz-representer for a linear functional
$l(u)$, then the functional of the error in the PPUM approximation
$u_{pp}$ from~(\ref{Holst:eqn:ppum_upp}) satisfies
$$
l(u-u_{pp}) 
    = - \sum_{k=1}^p \langle F(u_i^g), \omega_i \rangle,
$$
where $u_i^g$ are the solutions to the subspace problems
in~(\ref{Holst:eqn:ppum_uig}), and
where the $\omega_i$ are the solutions to the following
{\em global dual problems} with localized data:
\begin{equation}
   \label{Holst:eqn:dualsol}
\text{Find}~ \omega_i \in H^1_0(\Omega)~\text{such~that}~
 (A^T \omega_i, v)_{L^2(\Omega)}
                = (\phi_i \psi, v)_{L^2(\Omega)},
    ~~~\forall v \in H^1_0(\Omega).
\end{equation}
Moreover, if the local residual $F(u_i^g)$, weighted by the
localized dual solution $\omega_i$, satisfies the following error
tolerance in each subspace:
\begin{equation}
   \label{Holst:eqn:tol_local}
    \left| \langle F(u_i^g), \omega_i \rangle \right|
   < \frac{\epsilon}{p}, ~~~~ i=1,\ldots,p
\end{equation}
then the linear functional of the global error $u-u_{pp}$ satisfies
\begin{equation}
   \label{Holst:eqn:tol_global}
\left|l(u-u_{pp})\right| < \epsilon.
\end{equation}
\end{theorem}
\begin{proof}
With $l(u-u_{pp}) = (u-u_{pp}, \psi )_{L^2(\Omega)}$,
the localized representation comes from:
$$
   (u - u_{pp},\psi)_{L^2(\Omega)}
      = (\sum_{k=1}^p \phi_i u 
             - \sum_{i=1}^p \phi_i u_i^g,\psi)_{L^2(\Omega)}
      = \sum_{k=1}^p (\phi_i (u - u_i^g), \psi)_{L^2(\Omega \cap \Omega_i)}.
$$
From~(\ref{Holst:eqn:residual}) and~(\ref{Holst:eqn:dualsol}),
each term in the sum can be written
in terms of the local residual $F(u_i^g)$ as follows:
\begin{eqnarray*}
(\phi_i(u-u_i^g), \psi)_{L^2(\Omega \cap \Omega_i)}
  & = & (u-u_i^g, \phi_i \psi)_{L^2(\Omega \cap \Omega_i)} \\
  & = & (u-u_i^g, \mathcal{A}^T \omega_i)_{L^2(\Omega)} \\
  & = & (\mathcal{A}(u-u_i^g), \omega_i)_{L^2(\Omega)} \\
  & = & - ( F(u_i^g), \omega_i )_{L^2(\Omega)}.
\end{eqnarray*}
This gives then
$$
|(u-u_{pp},\psi)_{L^2(\Omega)}|
      \le \sum_{k=1}^p | \langle F(u_i^g), \psi \rangle|
      < \sum_{k=1}^p \frac{\epsilon}{p}
      = \epsilon.
$$
\end{proof}
We will make a few additional remarks about the parallel adaptive
algorithm which arises naturally from Theorem~\ref{Holst:thm:ppum_dual}.
Unlike the case in Theorem~\ref{Holst:thm:ppum}, the constants
$C_{\infty}$ and $C_G$ in~(\ref{Holst:eqn:pum_pou4})
and~(\ref{Holst:eqn:pum_pou5}) do not impact the error estimate
in Theorem~\ref{Holst:thm:ppum_dual}, removing the need for the
{\em a~priori} large overlap assumptions.
Moreover, local {\em a~priori} estimates are not required either,
removing a second separate large overlap assumption that must be made
to prove results such as Theorem~\ref{Holst:thm:ppum}.
Using large overlap of {\em a~priori} unknown size to satisfy the requirements
for Theorem~\ref{Holst:thm:ppum} seems unrealistic for implementations.
On the other hand, no such {\em a~priori} assumptions are
required to use the result in Theorem~\ref{Holst:thm:ppum_dual} as
the basis for a parallel adaptive algorithm.
One simply solves the local dual problems~(\ref{Holst:eqn:dualsol})
on each processor independently, adapts the mesh on each processor
independently until the computable local error estimate satisfies
the tolerance~(\ref{Holst:eqn:tol_local}), which
then guarantees that the functional of the global error meets
the target in~(\ref{Holst:eqn:tol_global}).

Whether such a duality-based approach will produce an efficient
parallel algorithm is not at all clear; however, it is at least
a mechanism for decomposing the solution to an elliptic problem 
over a number of subdomains.
Note that ellipticity is not used in Theorem~\ref{Holst:thm:ppum_dual},
so that the approach is also likely reasonable for other classes of PDE.
These questions, together with a number of related duality-based
decomposition algorithms are examined in more detail in~\cite{Estep:2002:SDL}.
The analysis in~\cite{Estep:2002:SDL} is based on a different approach
involving estimates of Green function decay rather than through
partition of unity methods.

\section[Implementation in {\sc FEtk} and {\sc MC}]
        {Implementation in {\sc FEtk} and {\sc MC}}
  \label{Holst:sec:mc}

Our implementations are performed using
{\sc FEtk} and {\sc MC} (see~\cite{Holst:2001:ANT}
for a more complete discussion of {\sc MC} and {\sc FEtk}).
{\sc MC} is the adaptive multilevel finite element software kernel
within {\sc FEtk}, a large collection of collaboratively developed
finite element software tools based at UC San Diego (see {\tt www.fetk.org}).
{\sc MC} is written in ANSI C (as is most of {\sc FEtk}),
and is designed to produce highly accurate numerical solutions to
nonlinear covariant elliptic systems of tensor equations
on 2- and 3-manifolds in an optimal or nearly-optimal way.
{\sc MC} employs {\em a posteriori} error estimation,
adaptive simplex subdivision, 
unstructured algebraic multilevel methods, global inexact Newton methods, 
and numerical continuation methods.
Several of the features of {\sc MC} are somewhat unusual,
allowing for the treatment
of very general nonlinear elliptic systems of tensor equations on domains
with the structure of (Riemannian) 2- and 3-manifolds.
Some of these features are:
\begin{itemize}
\item {\em Abstraction of the elliptic system}: The elliptic system 
     is defined only through a nonlinear weak form over the domain manifold,
     along with an associated linearization form, also defined everywhere on
     the domain manifold
     (precisely the forms $\langle F(u),v \rangle$
     and $\langle DF(u)w,v \rangle$ in the discussions above).
\item {\em Abstraction of the domain manifold}: The domain manifold is
     specified by giving a polyhedral representation of the topology, along
     with an abstract set of coordinate labels of the user's interpretation,
     possibly consisting of multiple charts.
     {\sc MC} works only with the topology of the domain, the connectivity
     of the polyhedral representation.
     The geometry of the domain manifold is provided only through the form
     definitions, which contain the manifold metric information.
\item {\em Dimension independence}: Exactly the same code paths in {\sc MC}
     are taken for both two- and three-dimensional problems (as well as for
     higher-dimensional problems).
     To achieve this dimension independence, {\sc MC} employs the simplex
     as its fundamental geometrical object for defining finite element bases.
\end{itemize}
As a consequence of the abstract weak form approach to defining the problem,
the complete definition of a complex nonlinear tensor system such as large
deformation nonlinear elasticity requires writing only a few hundred lines
of C to define the two weak forms.
Changing to a different tensor system (e.g. the example later in the
paper involving the constraints in the Einstein equations)
involves providing only a different definition of the forms and a different
domain description.

A datastructure referred to as the {\em ringed-vertex}
(cf.~\cite{Holst:2001:ANT})
is used to represent meshes of $d$-simplices of arbitrary topology.
This datastructure is illustrated in Figure~\ref{Holst:fig:river}.
\begin{figure}[ht]
\begin{center}
\mbox{\includegraphics[width=0.38\textwidth]{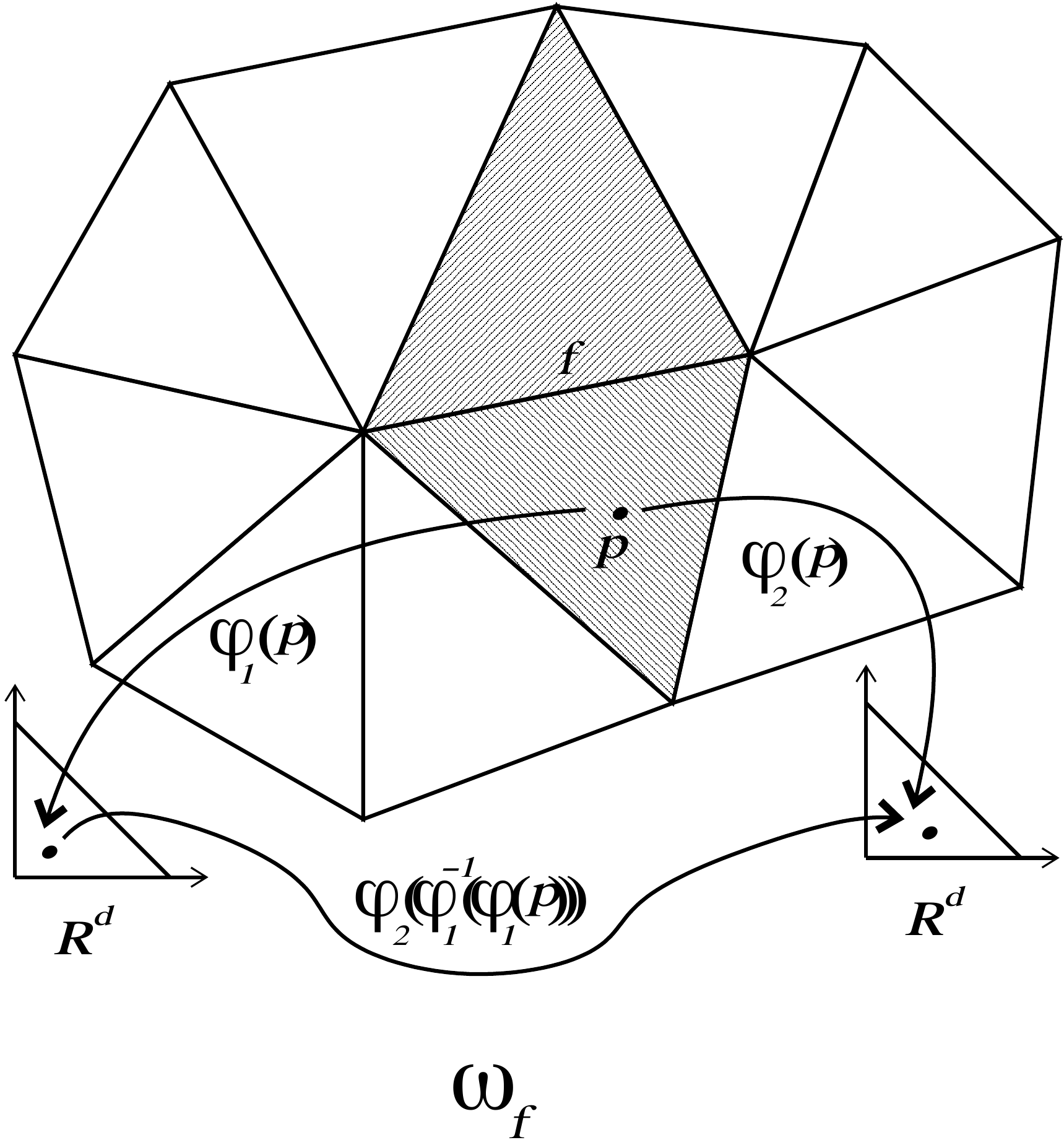}}
\mbox{ \hspace*{0.7cm} }
\mbox{\includegraphics[width=0.38\textwidth]{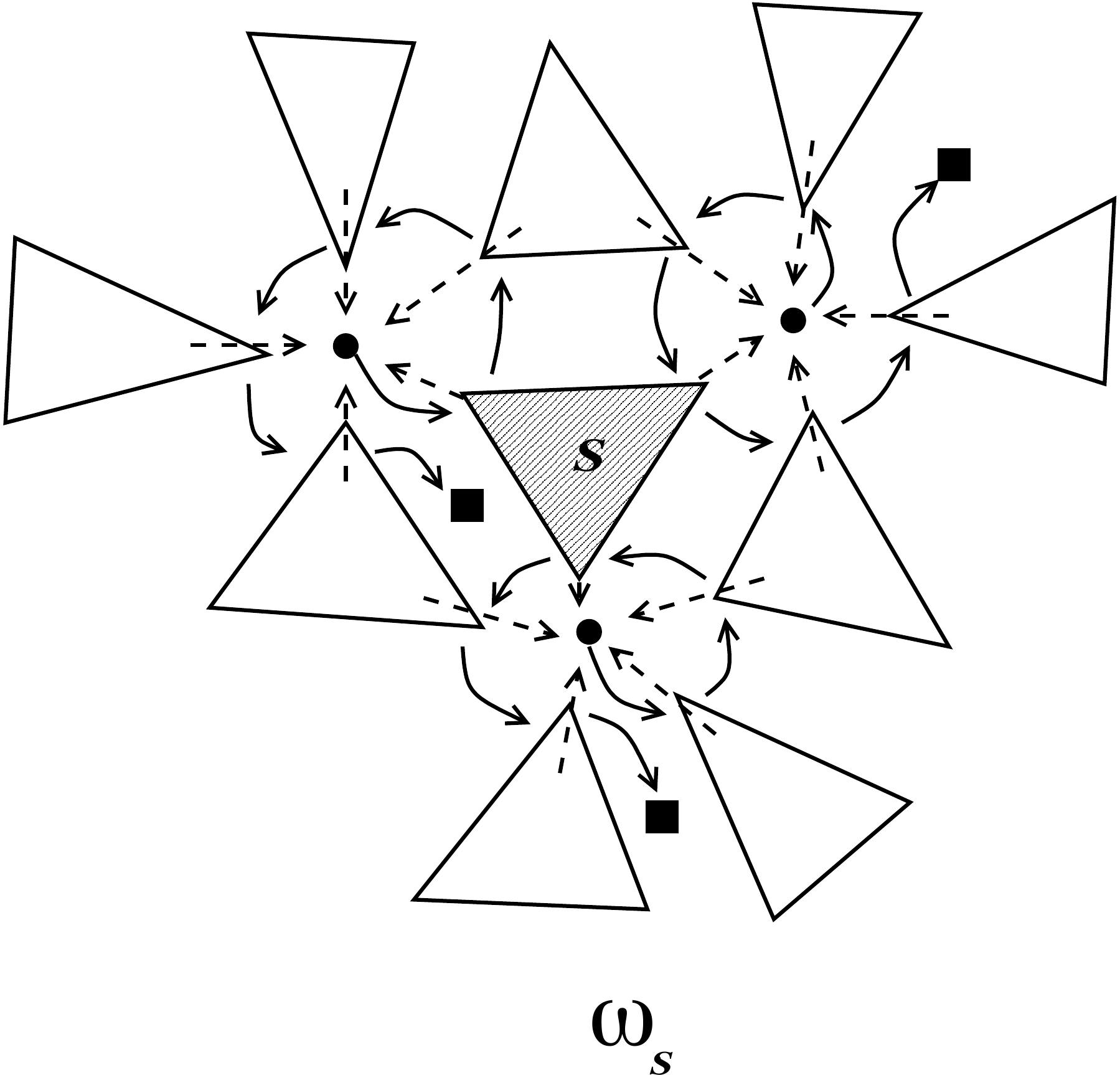}}
\end{center}
\caption{Polyhedral manifold representation.
The figure on the left shows two overlapping polyhedral (vertex) charts
consisting of the two rings of simplices around two vertices sharing an edge.
The region consisting of the two darkened triangles around the face $f$ is
denoted $\omega_f$, and represents the overlap of the two vertex charts.
Polyhedral manifold topology is represented by MC using the
{\em ringed-vertex} (or {\em RIVER}) datastructure.
The datastructure is illustrated for a given simplex $s$ in the figure on the
right; the topology primitives are vertices and $d$-simplices.
The collection of the simplices which meet the simplex $s$ at its vertices
(which then includes those simplices that share faces as well) is denoted
as $\omega_s$.
}
\label{Holst:fig:river}
\end{figure}
The ringed-vertex datastructure
is similar to the winged-edge, quad-edge, and edge-facet
datastructures commonly used in the computational geometry community for
representing 2-manifolds~\cite{Mucke:1993:SIT},
but it can be used more generally
to represent arbitrary $d$-manifolds, $d \ge 2$.
It maintains a mesh of $d$-simplices with near minimal storage,
yet for shape-regular (non-degenerate) meshes, it provides $O(1)$-time access
to all information necessary for refinement, un-refinement, and
Petrov-Galerkin discretization of a differential operator.
The ringed-vertex datastructure also allows for dimension independent
implementations of mesh refinement and mesh manipulation, with one
implementation (the same code path) covering arbitrary dimension $d$.
An interesting feature of this datastructure is that the C structures used
for vertices, simplices, and edges are all of fixed size, so that a fast
array-based implementation is possible, as opposed to a less-efficient
list-based approach commonly taken for finite element implementations
on unstructured meshes.
A detailed description of the ringed-vertex datastructure, along with a
complexity analysis of various traversal algorithms, can be found
in~\cite{Holst:2001:ANT}.

Our modifications to {\sc MC} to implement PPUM are minimal,
and are described in detail in~\cite{Bank:2000:NPP}.
These modifications involve primarily forcing the error indicator
to ignore regions outside the subdomain assigned to the particular
processor.
The implementation does not form an explicit partition of unity or a
final global solution; the solution must
be evaluated locally by locating the disjoint subdomain containing the
physical region of interest, and then by using the solution produced
by the processor assigned to that particular subdomain.
Note that forming a global conforming mesh as needed to build a global
partition of unity is possible even in a very loosely coupled parallel
environment, due to the deterministic nature of the bisection-based
algorithms we use for simplex subdivision (see~\cite{Holst:2001:ANT}).
For example, if bisection by longest edge (supplemented with tie-breaking)
is used to subdivide any simplex that is refined on any processor, then
the progeny types, shapes, and configurations can be predicted in a completely
determinstic way.
If two simplices share faces across a subdomain boundary, then they are
either compatible (their triangular faces exactly match), or one of the
simplices has been bisected more times than its neighbor.
By exchanging only the generation numbers between subdomains, a global
conforming mesh can be reached using only additional bisection.

\section[Example 1: The Einstein Constraints]
        {Example 1: The Einstein Constraints in Gravitation}
  \label{Holst:sec:example1}

The evolution of the gravitational field was conjectured by Einstein
to be governed by twelve coupled first-order hyperbolic equations for the
metric of space-time and its time derivative, where the evolution is
constrained for all time by a coupled four-component elliptic system.
The theory basically gives what is viewed as the correct interpretation
of the graviational field as a bending of space and time around matter
and energy, as opposed to the classical Newtonian view of the
gravitational field
as analogous to the electrostatic field; cf. Figure~\ref{Holst:fig:grpic}.
\begin{figure}[ht]
\centerline{
\hbox{\includegraphics[width=0.5\textwidth]{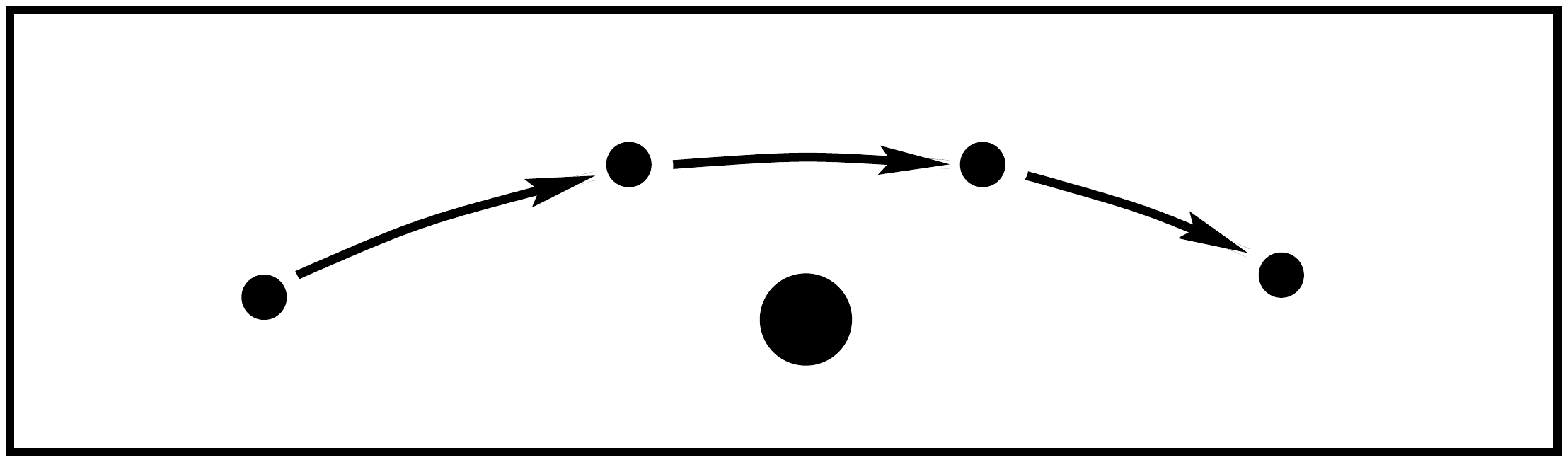}}
\hbox{\includegraphics[width=0.5\textwidth]{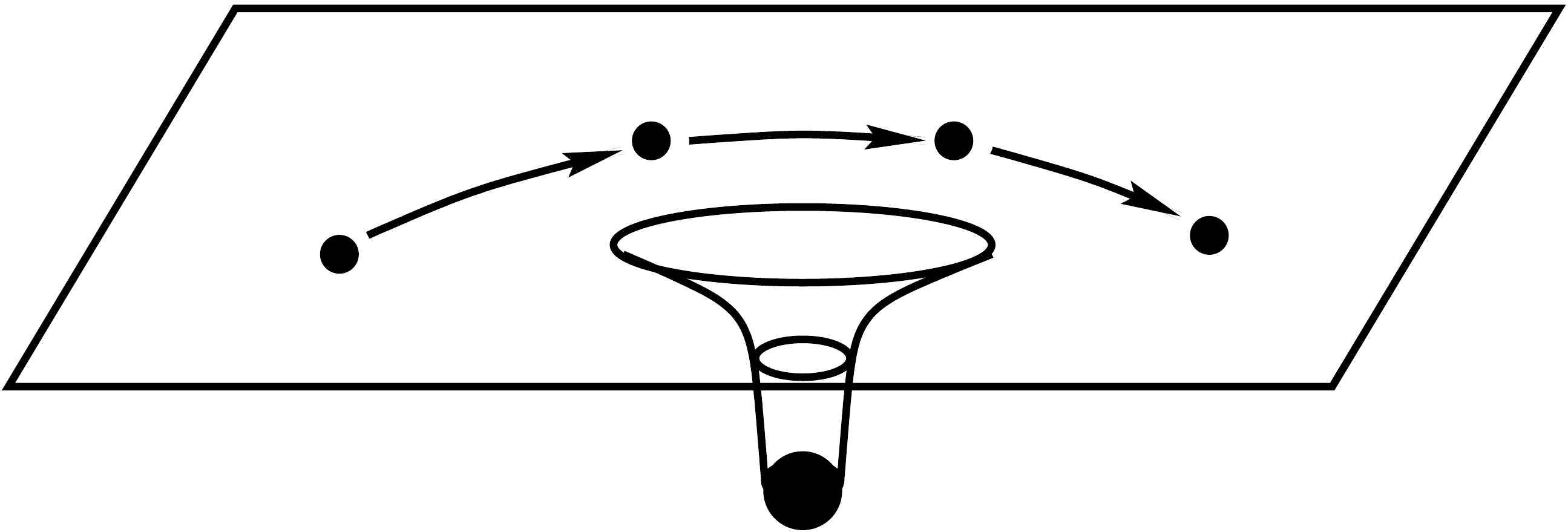}}
}
\caption{Newtonian versus general relativistic explanations of
gravitation: the small mass simply follows a geodesic on the curved
surface created by the large mass.}
\label{Holst:fig:grpic}
\end{figure}
The four-component elliptic constraint system consists of a nonlinear scalar
{\em Hamiltonian constraint}, and a linear 3-vector {\em momentum}
constraint.
The evolution and constraint equations, similar in some respects to
Maxwell's equations, are collectively referred to as
the {\em Einstein equations}.
Solving the constraint equations numerically, separately or together
with the evolution equations, is currently of great interest
to the physics community due to the recent construction of a new
generation of gravitational wave detectors
(cf.~\cite{Holst:2002:SRN,Holst:2002:AFE} for more detailed
discussions of this application).

Allowing for both Dirichlet and Robin boundary conditions on a
3-manifold $\mathcal{M}$ with boundary
$\partial \mathcal{M} = \partial_0 \mathcal{M} \cup \partial_1 \mathcal{M}$,
as typically the case in black hole and neutron star models
(cf.~\cite{Holst:2002:SRN,Holst:2002:AFE}), the strong form of the constraints
can be written as:
\begin{eqnarray}
\hat{\Delta}\phi & = & \frac{1}{8} \hat{R} \phi
     + \frac{1}{12} ({\rm tr} K)^2 \phi^5
\label{Holst:eqn:ham_str1} \\
& &
     - \frac{1}{8} ({{}^{*}\!\!\hat{A}}_{ab} + (\hat{L}W)_{ab})^2 \phi^{-7}
     - 2 \pi \hat{\rho} \phi^{-3} ~~\text{in}~\mathcal{M},
\nonumber \\
\hat{n}_a \hat{D}^a \phi + c \phi
     & = & z ~\text{on}~\partial_1 \mathcal{M}, \label{Holst:eqn:ham_str2} \\
\phi & = & f ~\text{on}~\partial_0 \mathcal{M}, \label{Holst:eqn:ham_str3} \\
  \hat{D}_b(\hat{L}W)^{ab} & = & \frac{2}{3} \phi^6 \hat{D}^a {\rm tr}K 
                             + 8 \pi \hat{j}^a
      ~~\text{in}~\mathcal{M}, \label{Holst:eqn:mom_str1} \\
(\hat{L}W)^{ab} \hat{n}_b + C^a_{~b} W^b
    & = & Z^a ~\text{on}~\partial_1 \mathcal{M}, \label{Holst:eqn:mom_str2} \\
  W^a & = & F^a ~~\text{on}~\partial_0 \mathcal{M},  \label{Holst:eqn:mom_str3}
\end{eqnarray}
where the following standard notation has been employed:
\begin{eqnarray*}
\hat{\Delta}\phi & = & \hat{D}_a \hat{D}^a \phi, \\
(\hat{L}W)^{ab} & = & \hat{D}^a W^b + \hat{D}^b W^a 
                - \frac{2}{3} \hat{\gamma}^{ab} \hat{D}_c W^c, \\
{\rm tr} K & = & \gamma^{ab}K_{ab}, \\
(C_{ab})^2 & = & C^{ab}C_{ab}.
\end{eqnarray*}
In the tensor expressions above, there is an implicit sum on all repeated
indices in products, and the covariant derivative with respect to the
fixed background metric $\hat{\gamma}_{ab}$ is denoted as $\hat{D}_{a}$
The remaining symbols in the equations
($\hat{R}$, $K$, ${{}^{*}\!\!\hat{A}}_{ab}$,
 $\hat{\rho}$, $\hat{j}^a$, $z$, $Z^a$, $f$, $F^a$, $c$, and $C^a_b$)
represent various physical parameters, and are described in detail
in~\cite{Holst:2002:SRN,Holst:2002:AFE} and the referenences therein.
Stating the system as set of tensor equations comes from the need
to work with domains which generally have the structure of 3-manifolds rather
than single open sets in $\mathbb{R}^3$ (cf.~\cite{Holst:2001:ANT}).

Equations (\ref{Holst:eqn:ham_str1})--(\ref{Holst:eqn:mom_str3})
are known to be well-posed only for certain problem data and
manifold
topologies~\cite{Murchadha:1973:EUS,Isenberg:1996:NMC}.
Note that if multiple solutions in the form of folds or
bifurcations are present in solutions of
(\ref{Holst:eqn:ham_str1})--(\ref{Holst:eqn:mom_str3}) then 
path-following numerical methods will be required for numerical
solution~\cite{Keller:1987:NMB}.
For our purposes here, we select the problem data and manifold
topology such that the assumptions for the two general well-posedness
results in~\cite{Holst:2002:SRN} hold for
(\ref{Holst:eqn:ham_str1})--(\ref{Holst:eqn:mom_str3}). 
The assumptions required for the two results in~\cite{Holst:2002:SRN}
are quite weak,
and are, for the most part, minimal assumptions beyond those required to give
a well-defined weak formulation in $L^p$-based Sobolev spaces.

In~\cite{Holst:2001:ANT}, two quasi-optimal {\em a~priori} error estimates
are established for Galerkin approximations to the solutions to
(\ref{Holst:eqn:ham_str1})--(\ref{Holst:eqn:mom_str3}). 
These take the form (see Theorems~4.3 and~4.4 in~\cite{Holst:2001:ANT}):
\begin{eqnarray}
\|u-u_h\|_{H^1(\mathcal{M})}
   & \le & C \inf_{v \in V_h} \|u-v\|_{H^1(\mathcal{M})}
\label{Holst:eqn:schatz} \\
\|u-u_h\|_{L^2(\mathcal{M})}
   & \le & C a_h \inf_{v \in V_h} \|u-v\|_{H^1(\mathcal{M})},
\end{eqnarray}
where $V_h \subset H^1(\mathcal{M})$ is e.g. a finite element space.
In the case of the momentum constraint, there is a restriction on the
size of the elements in the underlying finite element mesh for the above
results to hold, characterized above by the parameter $a_h$.
This restriction is due to the fact that the result is established through
of the G{\aa}rding inequality result due to Schatz~\cite{Schatz:1974:OCR}.
In the case of the Hamiltonian constraint, there are no restrictions
on the approximation spaces.

To use {\sc MC} to calculate the initial bending of space and time around a
single massive black hole by solving the above
constraint equations, we place a spherical object of unit radius in space,
and infinite space is truncated with an enclosing sphere of radius 100.
(This outer boundary may be moved further from the object to improve the
accuracy of boundary condition approximations.)
Reasonable choices for the remaining functions and parameters appearing
in the equations are used below to completely specify the problem for use
as an illustrative numerical example.
(More careful examination of the various functions and parameters 
appear in~\cite{Holst:2002:SRN}, and a number of detailed
experiments with more physically meaningful data appear
in~\cite{Holst:2002:AFE}.)

We then generate an initial (coarse) mesh of tetrahedra inside the enclosing
sphere, exterior to the spherical object within the enclosing sphere.
The mesh is generated by adaptively bisecting an initial mesh consisting
of an icosahedron volume filled with tetrahedra.
The bisection procedure simply bisects any tetrahedron which touches the
surface of the small spherical object.
When a reasonable approximation to the surface of the sphere is obtained,
the tetrahedra completely inside the small spherical object are removed,
and the points forming the surface of the small spherical object are
projected to the spherical surface exactly.
This projection involves solving a linear elasticity problem,
together with the use of a shape-optimization-based smoothing procedure.
The smoothing procedure locally optimizes the shape measure function
described in~\cite{Holst:2001:ANT} in an iterative fashion.
A much improved binary black hole mesh generator has been developed by
D.~Bernstein; the new mesh generator is described in~\cite{Holst:2002:AFE}
along with a number of more detailed examples using MC.

The initial coarse mesh is shown in Figure~\ref{Holst:fig:gr1}, generated using
the procedure described above, has approximately 30,000 tetrahedral elements
and 5,000 vertices.
To solve the problem on a 4-processor computing cluster using
a PPUM-like algorithm, we begin by partitioning the domain into
four subdomains (shown in Figure~\ref{Holst:fig:gr2})
with approximately equal error
using the recursive spectral bisection algorithm described
in~\cite{Bank:2000:NPP}.
The four subdomain problems are then solved independently by {\sc MC}, starting
from the complete coarse mesh and coarse mesh solution.
The mesh is adaptively refined in each subdomain until
a mesh with roughly 50000 vertices is obtained
(yielding subdomains with about 250000 simplices each).

The refinement performed by {\sc MC} is confined
primarily to the given region as driven by the weighted residual
error indicator described in~\cite{Holst:2001:ANT},
with some refinement into adjacent regions
due to the closure algorithm which maintains conformity and shape regularity. 
The four problems are solved completely independently by the
sequential adaptive software package {\sc MC}.
One component of the solution (the conformal factor $\phi$) of the
elliptic system is depicted in Figures~\ref{Holst:fig:gr3}
(the subdomain 0 and subdomain 1 solutions).

A number of more detailed examples involving the contraints,
using more physically meaningful data, appear in~\cite{Holst:2002:AFE}.
Application of PPUM to massively parallel simulations of microtubules
and other extremely large and complex biological structures can be
found in~\cite{Baker:2001:ENAM,Baker:2001:AMF}.
The results in~\cite{Baker:2001:ENAM,Baker:2001:AMF} demonstrate both good
parallel scaling of PPUM as well as quality approximation of
the gradient of electrostatic potentials (solutions to the 
Poisson-Boltzmann equation; cf.~\cite{Holst:2000:AMF}).

\begin{figure}[ht]
\centerline{
\mbox{\includegraphics[width=0.5\textwidth]{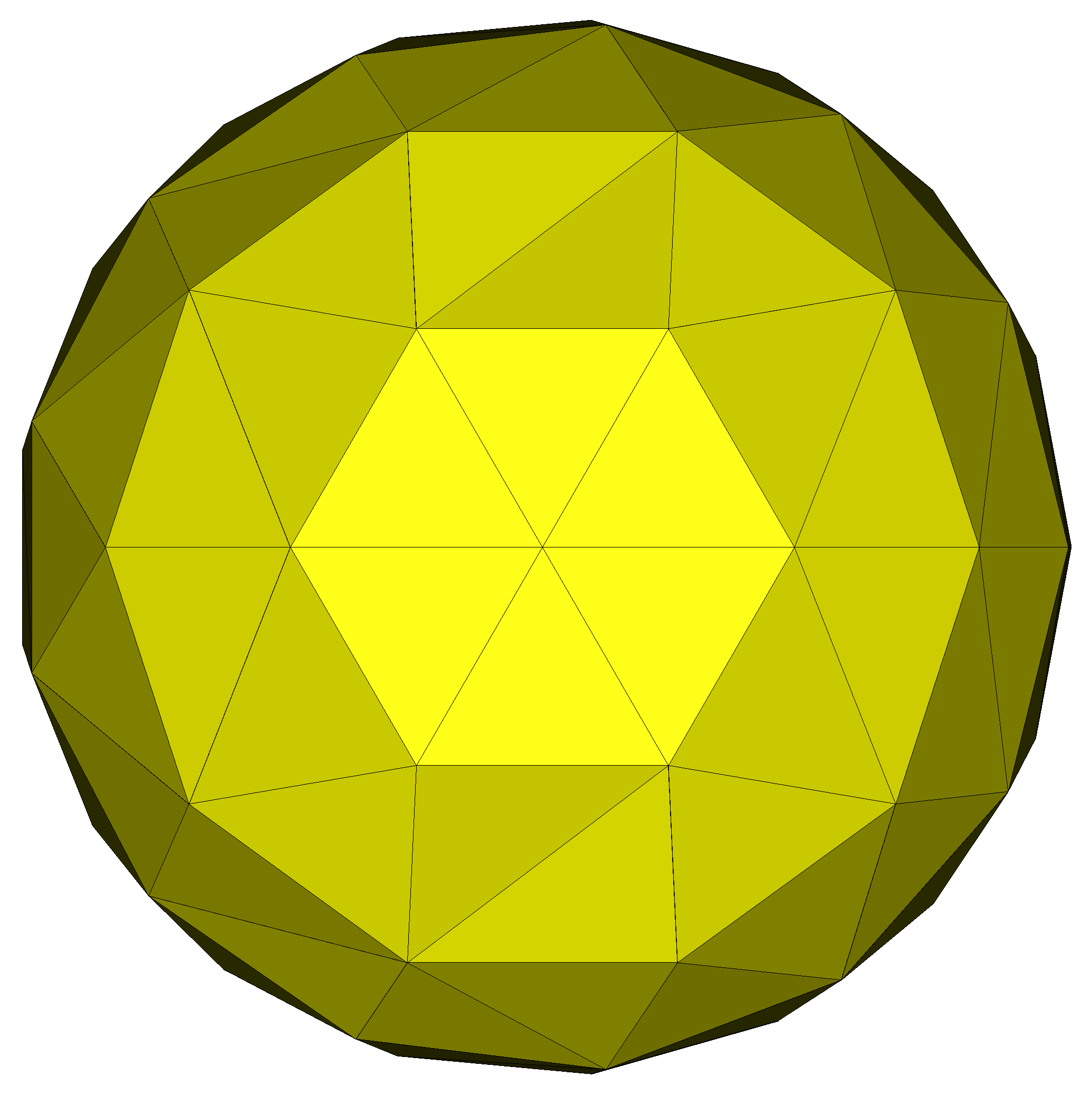}}
\mbox{\includegraphics[width=0.5\textwidth]{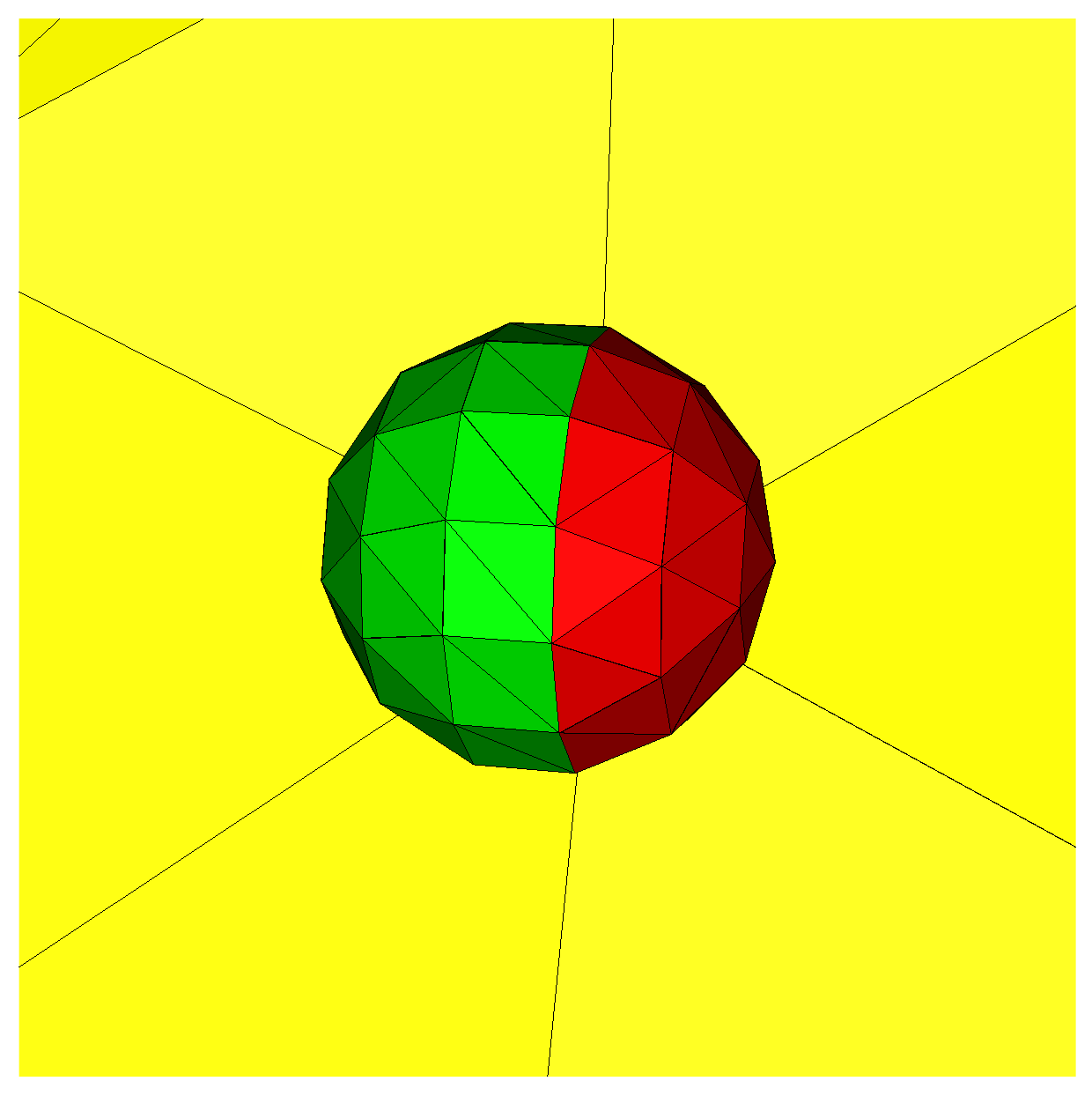}}
}
\caption{Recursize spectral bisection of the single hole domain
         into four subdomains (boundary surfaces of three of the
         four subdomains are shown).}
   \label{Holst:fig:gr1}
\end{figure}

\begin{figure}[ht]
\centerline{
\mbox{\includegraphics[width=0.5\textwidth]{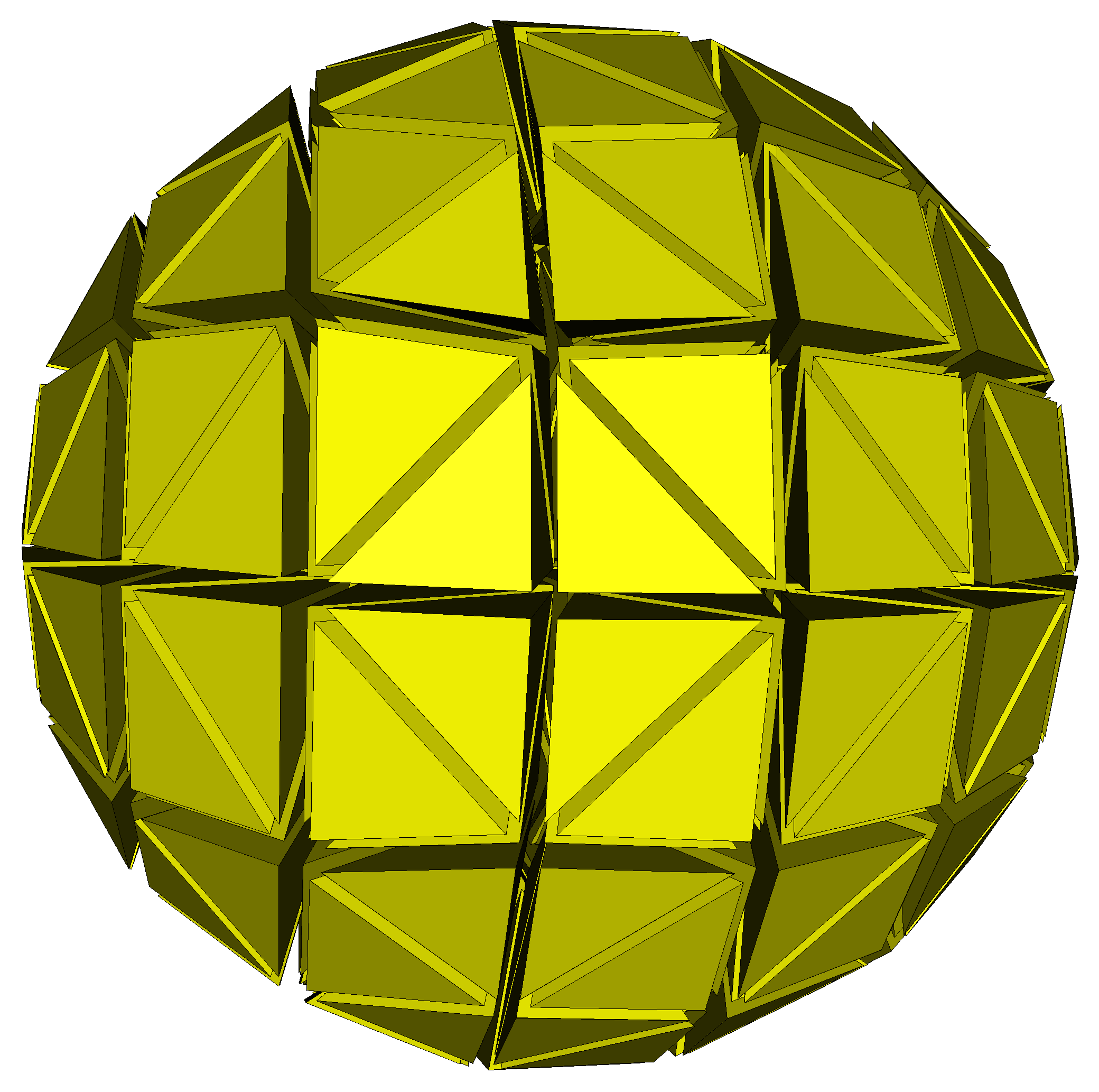}}
\mbox{\includegraphics[width=0.5\textwidth]{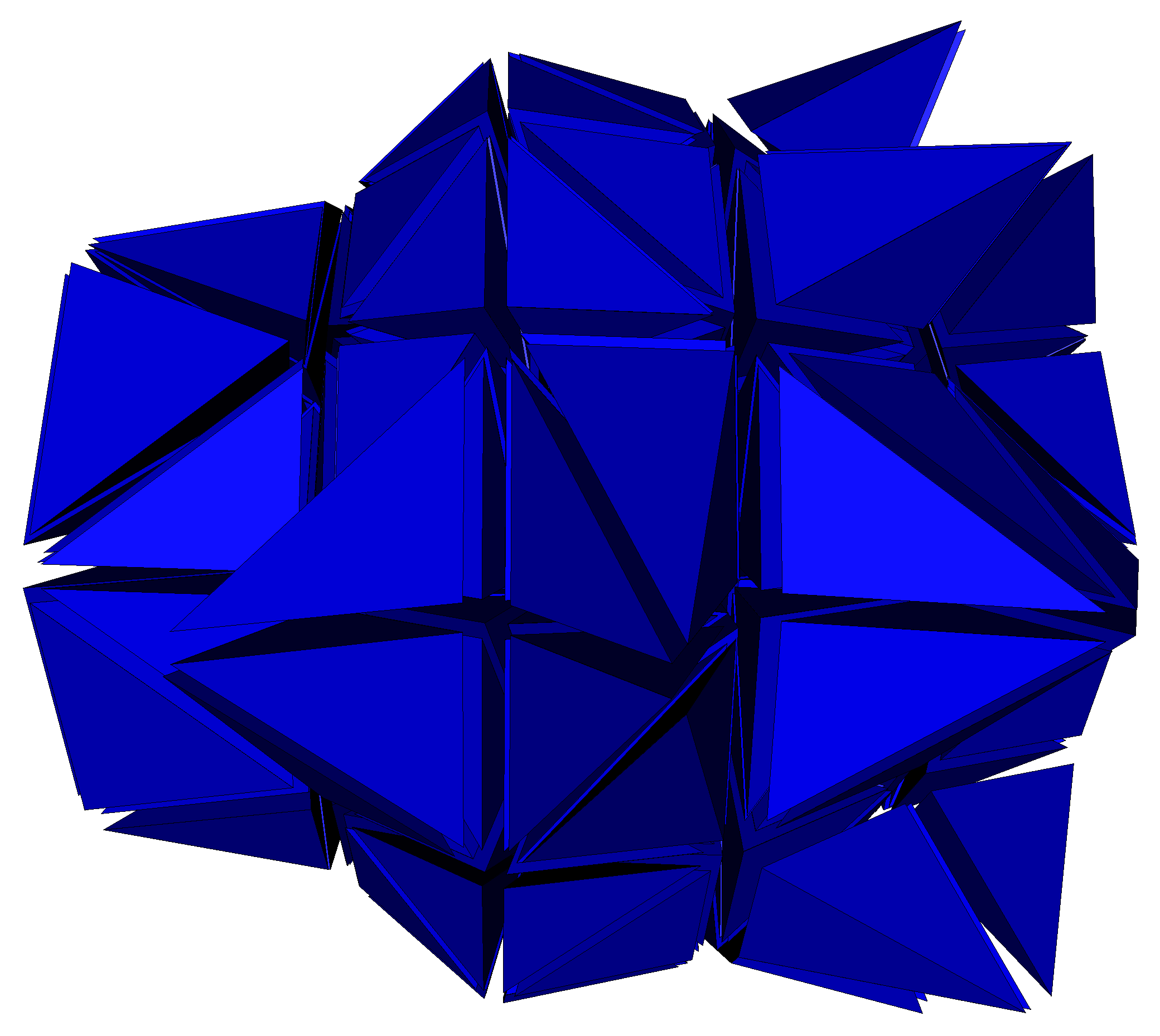}}
}
\centerline{
\mbox{\includegraphics[width=0.5\textwidth]{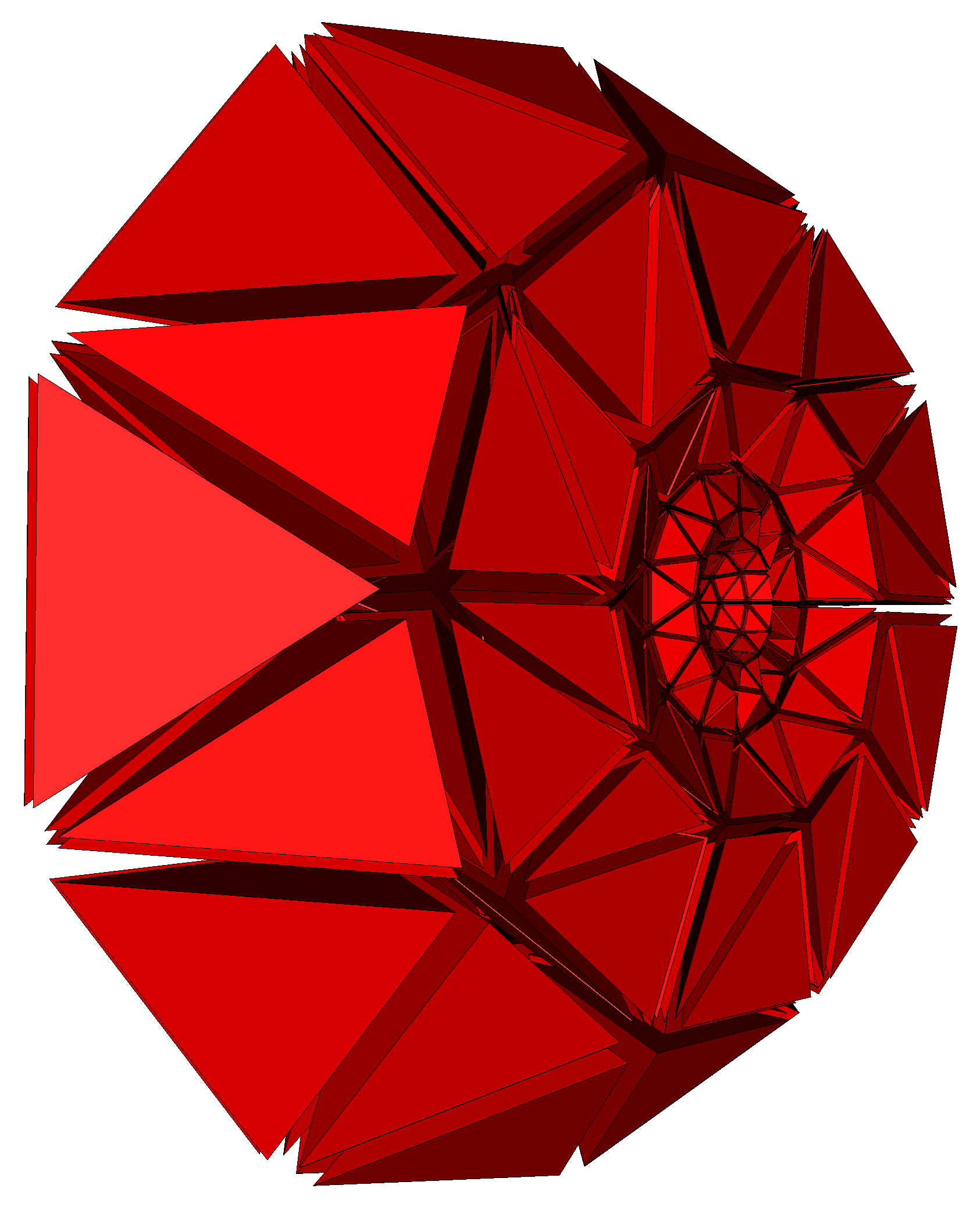}}
\mbox{\includegraphics[width=0.5\textwidth]{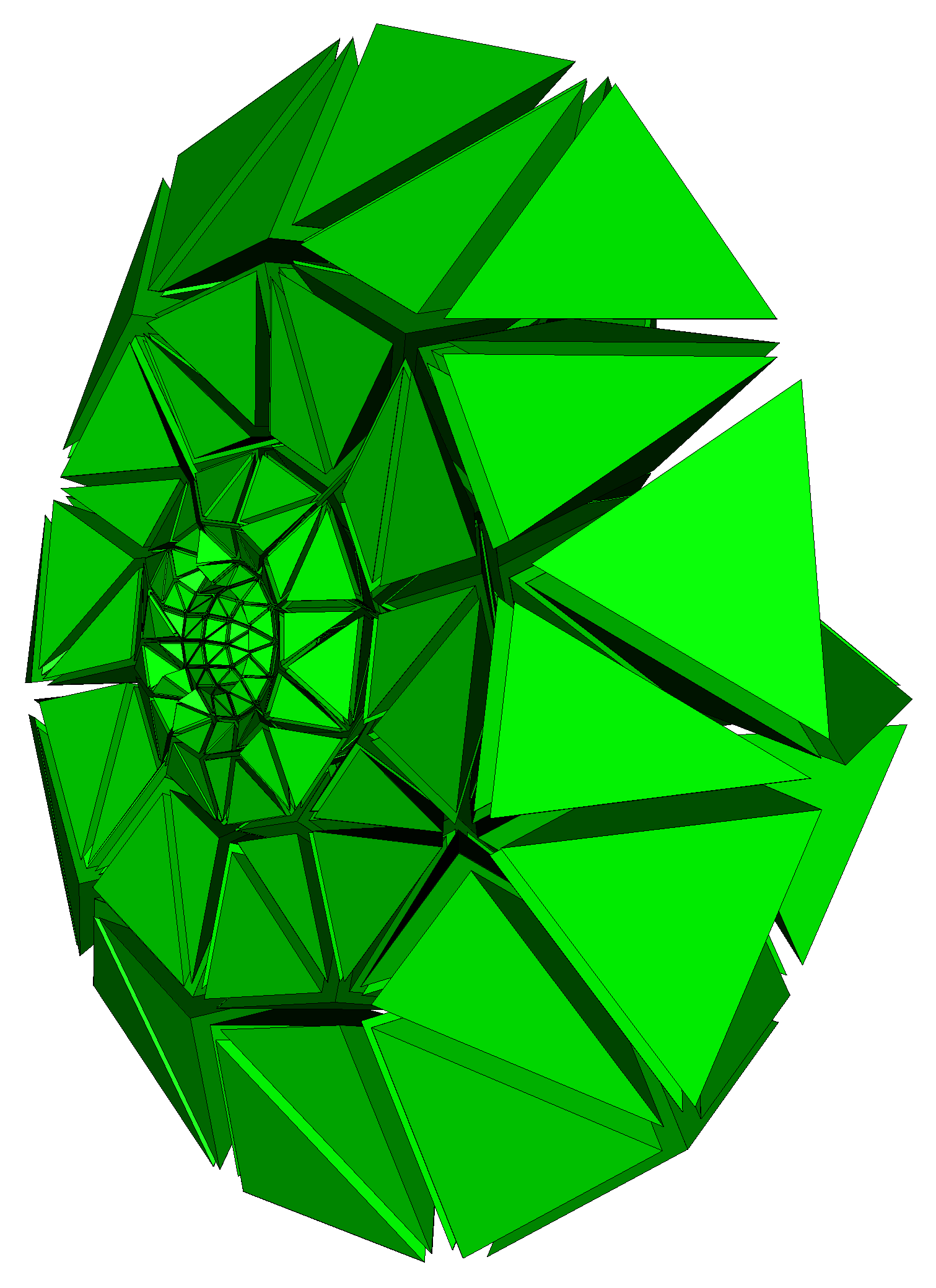}}
}
\caption{Recursize spectral bisection of the single hole domain
         into four subdomains.}
   \label{Holst:fig:gr2}
\end{figure}

\begin{figure}[ht]
\centerline{
\mbox{\includegraphics[width=0.5\textwidth]{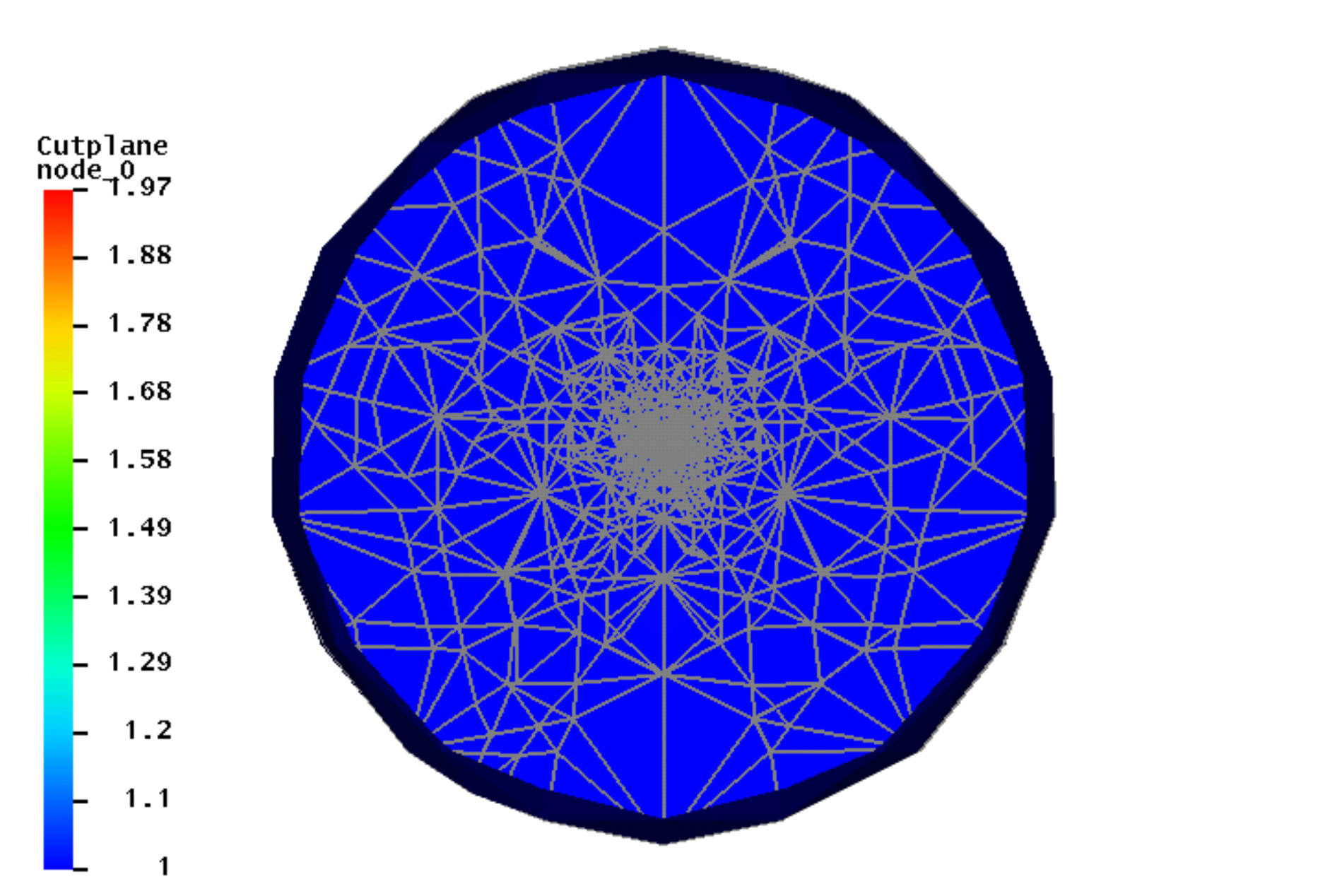}}
\mbox{\includegraphics[width=0.5\textwidth]{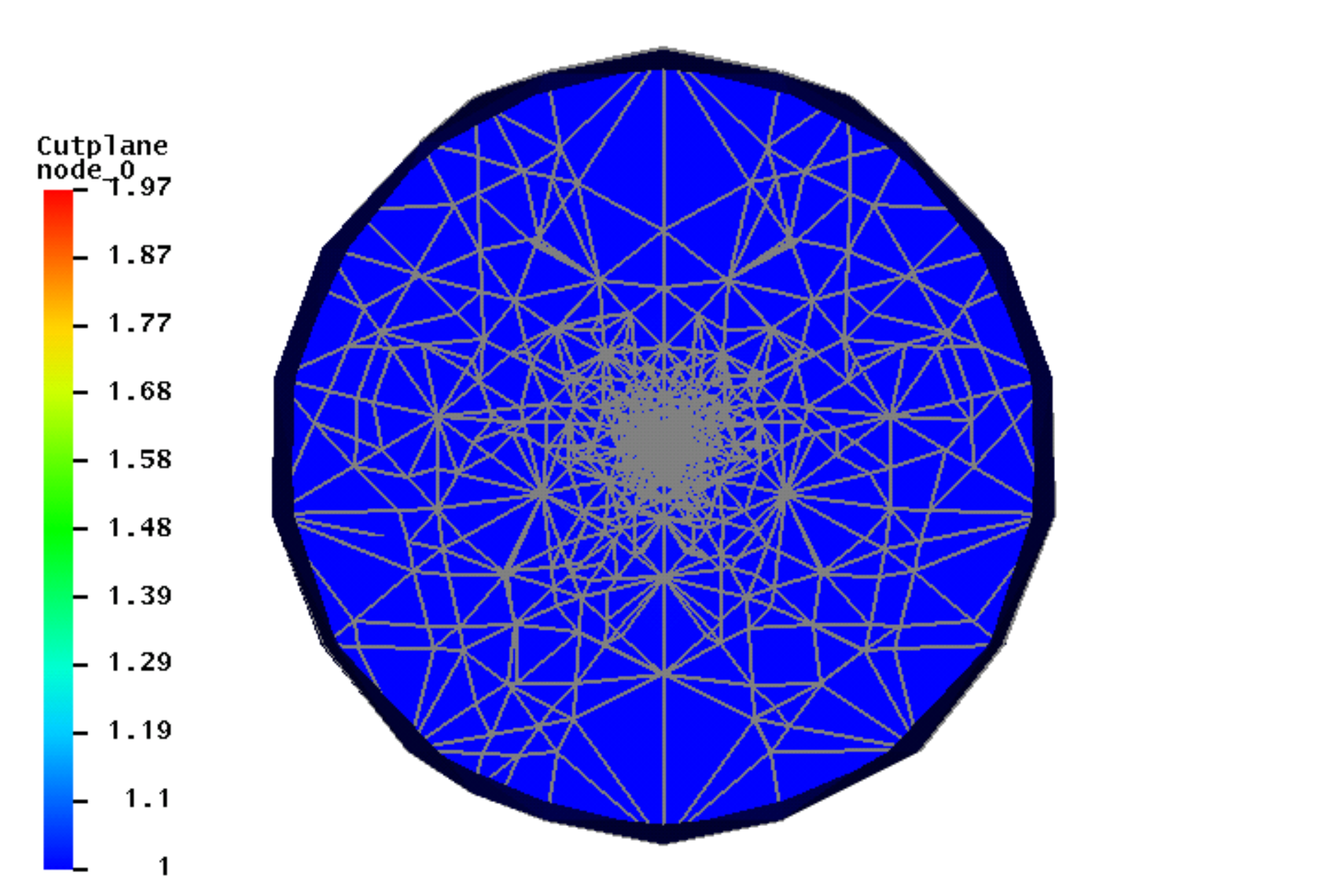}}
}
\centerline{
\mbox{\includegraphics[width=0.5\textwidth]{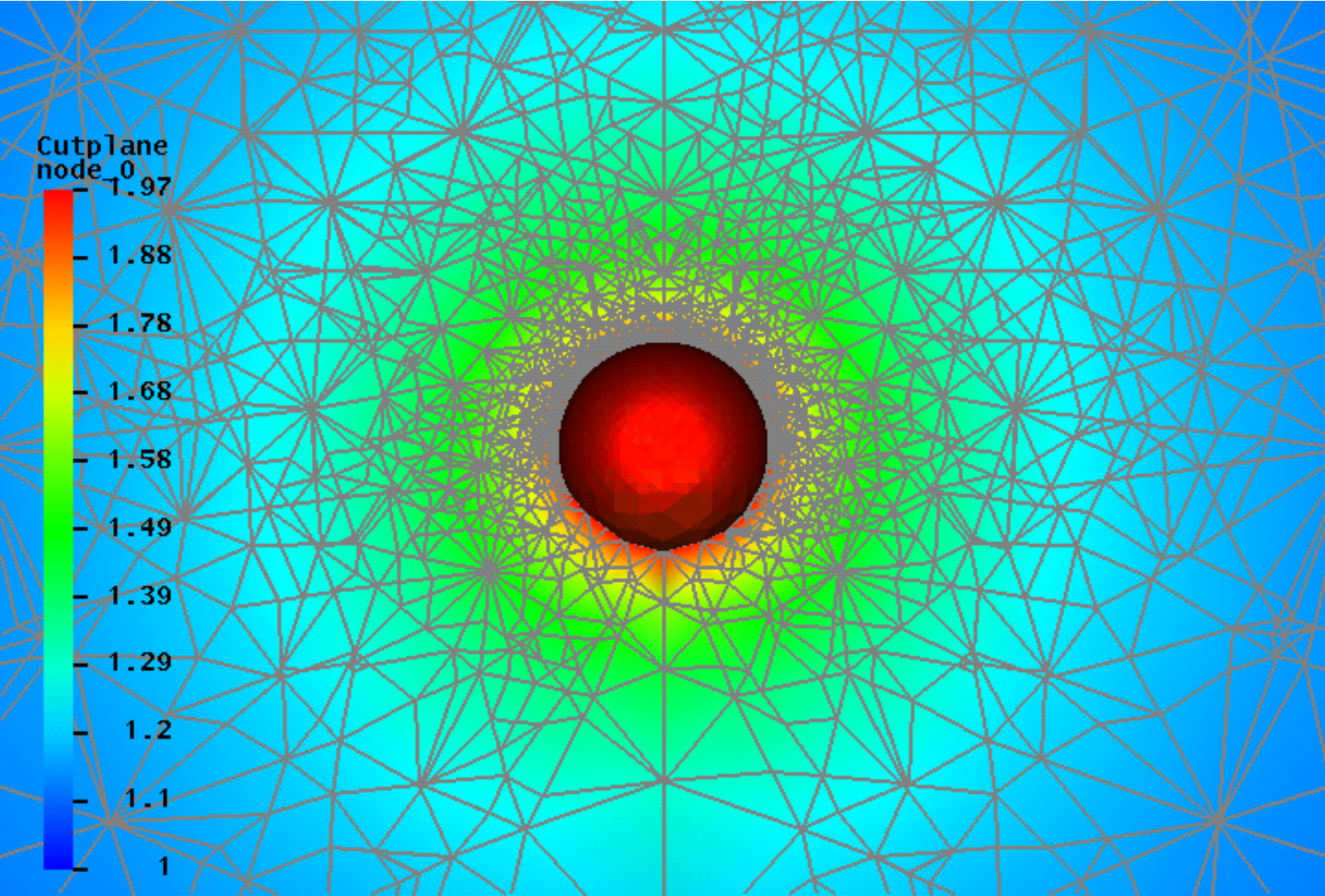}}
\mbox{\includegraphics[width=0.5\textwidth]{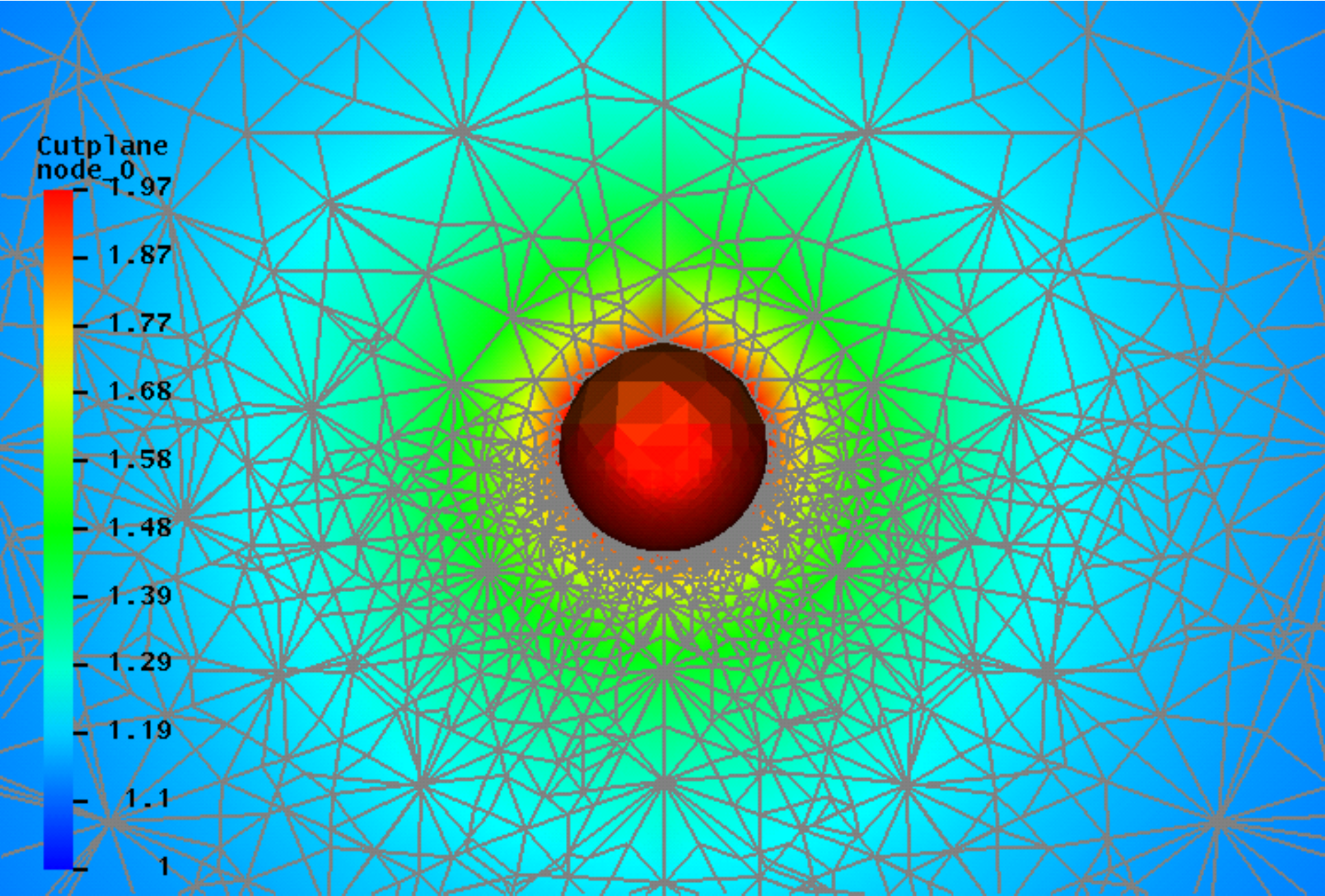}}
}
\centerline{
\mbox{\includegraphics[width=0.5\textwidth]{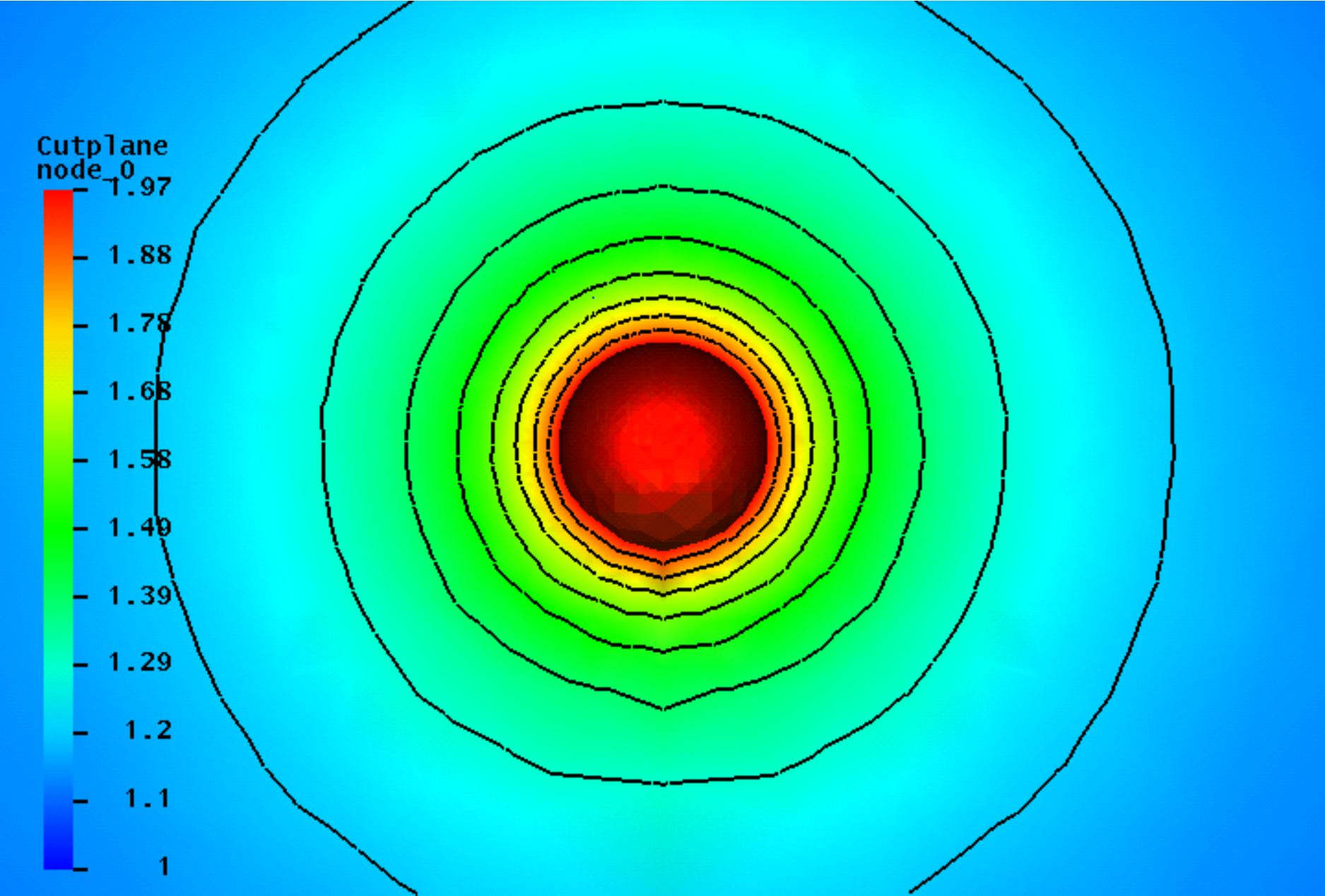}}
\mbox{\includegraphics[width=0.5\textwidth]{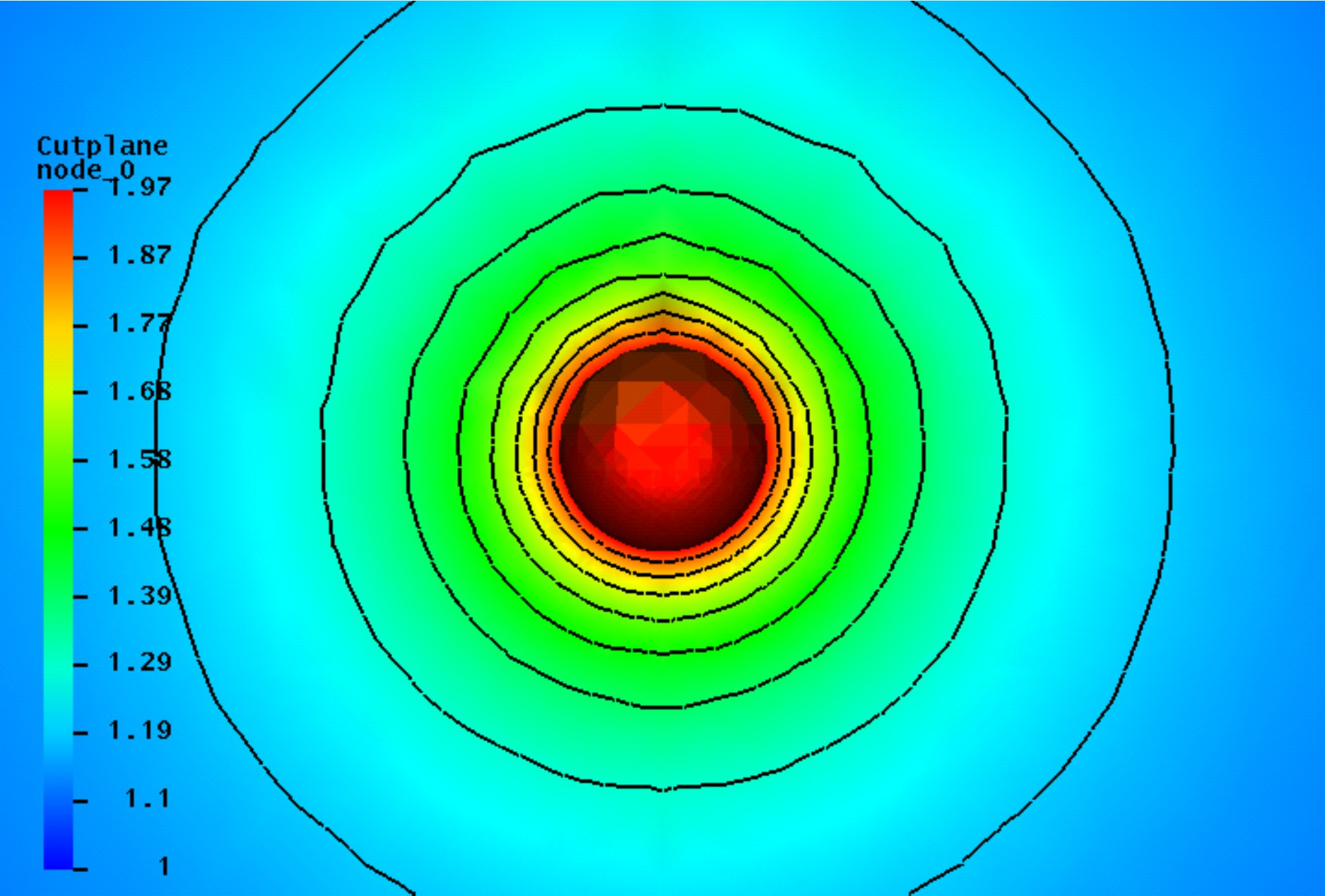}}
}
\centerline{
\mbox{\includegraphics[width=0.5\textwidth]{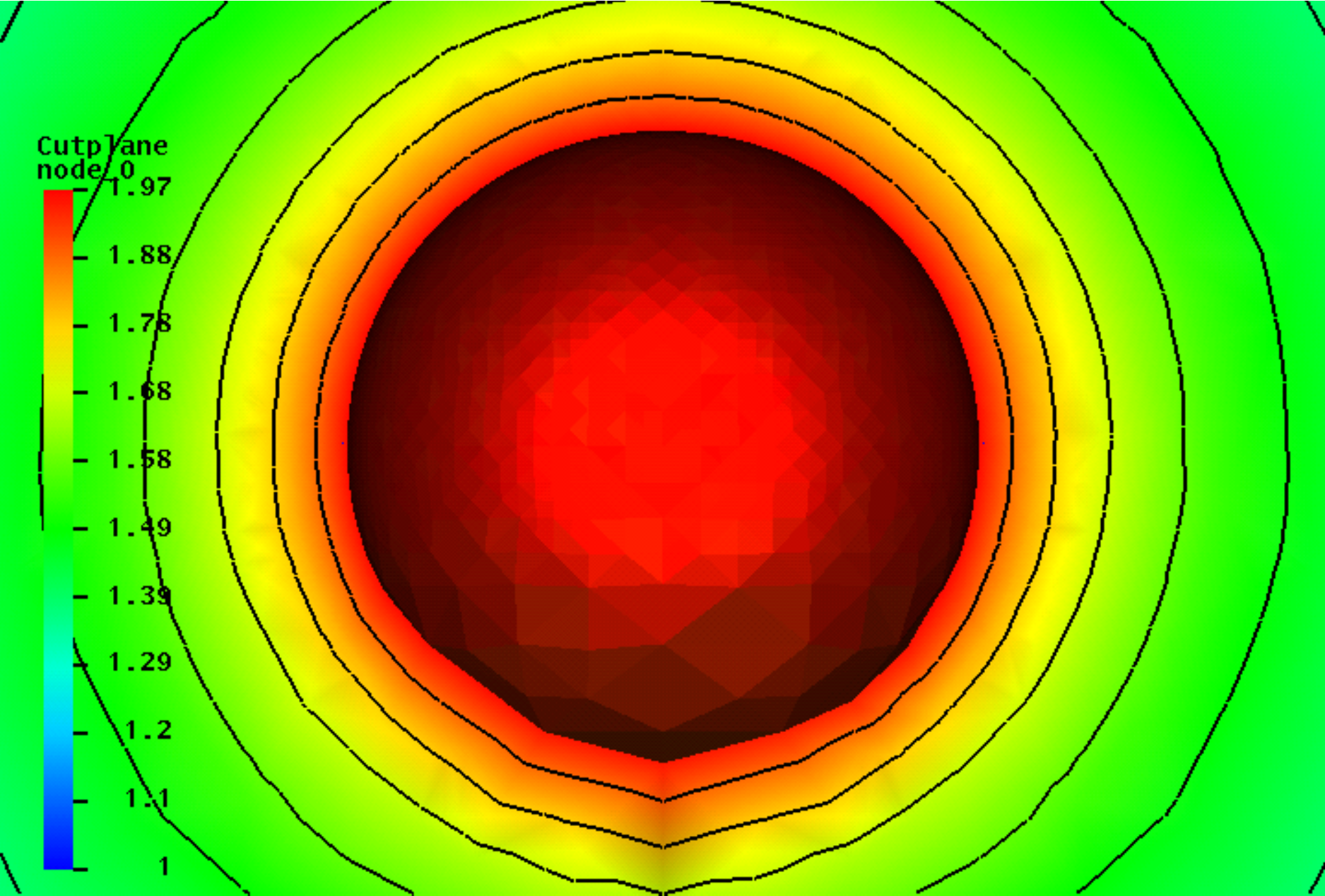}}
\mbox{\includegraphics[width=0.5\textwidth]{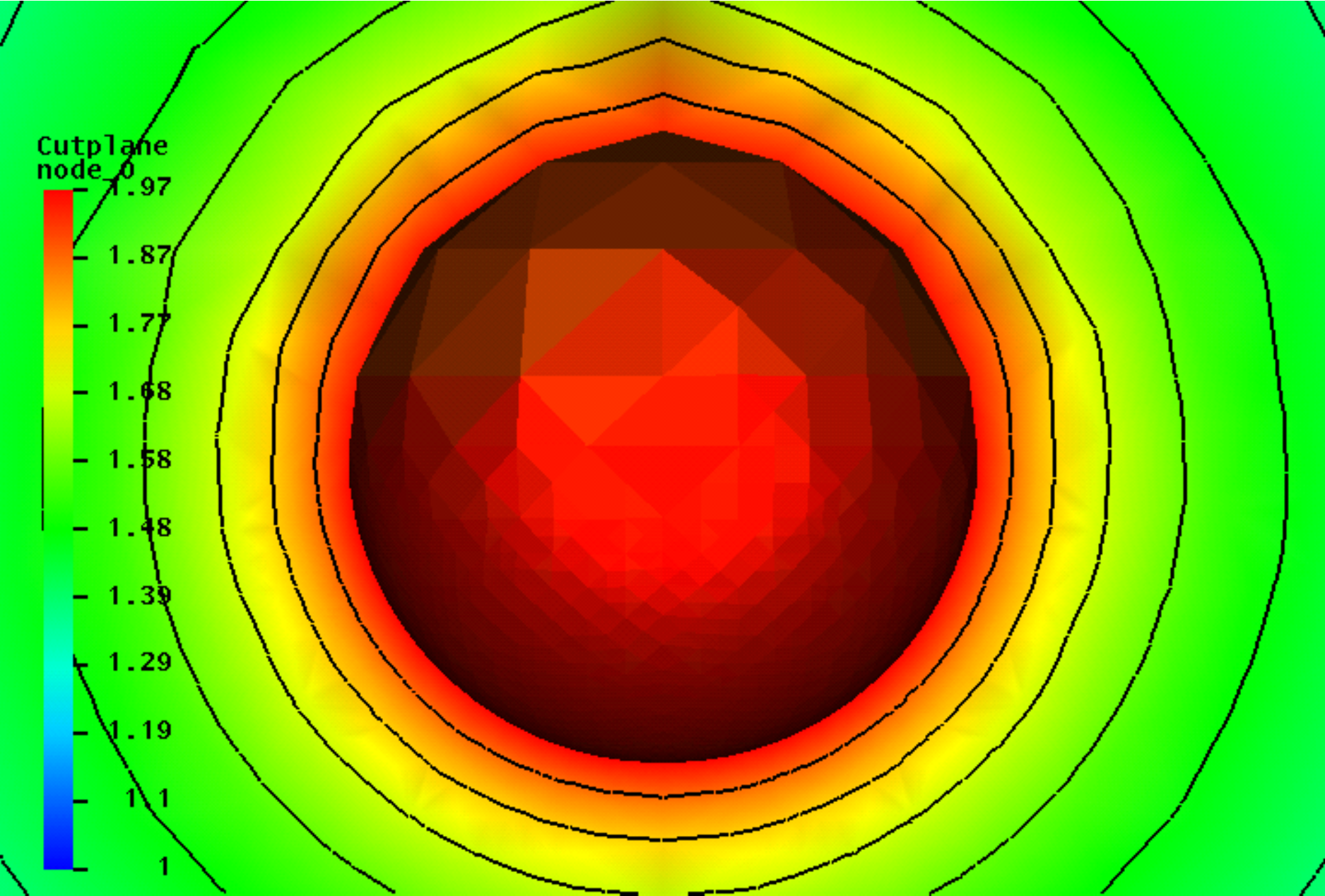}}
}
\caption{Decoupling of the scalar conformal factor in the initial data
         using PPUM; domain 0 in the left column, and domain 1 on the right.}
   \label{Holst:fig:gr3}
\end{figure}

\bibliographystyle{abbrv}
\bibliography{m}


\vspace*{0.5cm}

\end{document}